\documentclass[11pt,a4paper]{amsart}

\usepackage[T1]{fontenc}
\usepackage[latin1]{inputenc}
\usepackage{mathrsfs,hyperref,color} 

\newtheorem{theorem}{Theorem}
\newtheorem{lemma}{Lemma}
\newtheorem{corollary}{Corollary}

\newtheorem{definition}{Definition}
\newtheorem{remark}{Remark}
\newtheorem{example}{Example}

\DeclareMathOperator\Int{int}

\def\B{\mathbb B}
\def\C{\mathbb C} 
\def\N{\mathbb N}
\def\R{\mathbb R}
\def\T{\mathbb T}
\def\X{\mathbb X}

\def\cA{\mathcal A}
\def\cH{\mathcal H}
\def\cK{\mathcal K}
\def\cO{\mathcal O}
\def\cR{\mathcal R}

\def\scrB{\mathscr B} 
\def\scrC{\mathscr C} 
\def\scrD{\mathscr D} 
\def\scrE{\mathscr E} 
\def\scrH{\mathscr H} 
\def\scrL{\mathscr L} 
\def\scrM{\mathscr M}

\def\re{{\rm Re}}

\begin{document}

\title[Evolution Equations on Conic Manifolds]{Evolution Equations on Manifolds with Conical Singularities}

\author{Elmar Schrohe} 
\address{E. Schrohe, Leibniz University Hannover, Institute of Analysis, Welfengarten 1, 30167 Hannover, Germany }
\email{schrohe@math.uni-hannover.de} 

\begin{abstract}
This is an introduction to the analysis of nonlinear evolution  equations on manifolds with conical singularities via maximal regularity techniques. We address the specific difficulties due to the singularities, in particular the choice of extensions of the conic Laplacian that guarantee the existence of a bounded $H_\infty$-calculus. We introduce the relevant technical tools and survey, as main examples, applications to the porous medium equation, the fractional porous medium equation, the Yamabe flow, and the Cahn-Hilliard equation. 
\end{abstract}


\maketitle
\tableofcontents

%\numberwithin{equation}{section}

\date{\today}

%%%%%%%%%%%%%%%%%%%%%%%%
\section{Introduction} 
%%%%%%%%%%%%%%%%%%%%%%%%

Understanding the interplay of analysis and geometry has always been  a central issue in mathematics. 
 In the present article we consider nonlinear parabolic equations on manifolds that have conical singularities, see Section \ref{sect.2.1} for the precise definition. On the one hand we are interested in the existence of solutions in this context, on the other hand we would like to know how the geometry of the manifold near the singularities is reflected in the solutions.  

Concerning the existence of solutions, we rely on maximal regularity techniques. 
Maximal regularity has become an indispensable tool in the analysis of nonlinear parabolic problems. In connection with the Cl\'ement-Li theorem  it guarantees the existence of short-time solutions to initial value problems of the form 
\begin{eqnarray}
\label{1.1}
%\lefteqn{
\partial_t u + A(u)u =F(t,u), \quad u(0) = u_0,
\end{eqnarray}
under suitable conditions on the right hand side $F$ and the initial value $u_0$, provided that  $A(u(t))$ is, for each $t$, a closed  operator in the Banach space $\X_0$ with fixed domain $\scrD(A(u(t)))=\X_{1}$, where $\X_1$ is densely embedded in $\X_0$ and  $A(u_0)$ has maximal regularity,   see Theorem \ref{ThmCL}, below. 

A first issue are the closed extensions of cone differential operators. For elliptic operators on (domains in) $\R^n$ or on closed manifolds, the choice of the spaces $\X_0$ and $\X_1$ is clear in general, and maximal regularity has been established in many instances. On manifolds with conical singularities, the situation is more subtle. In general, the domain of the maximal extension differs from that of the minimal by a nontrivial finite-dimensional space. 
We review the corresponding results and show, how this finite-dimensional space can be determined conveniently.  In view of the applications we have in mind, our focus lies on the cone Laplacian. 
Here we show, moreover, how extensions can be chosen in such a way that the Laplacian admits a bounded $H_\infty$-calculus. 

We then consider examples, namely the porous medium equation (PME), the fractional PME, the Yamabe problem, and the Cahn-Hilliard equation.

The PME, also called nonlinear heat equation, models -- among other things -- the flow of a gas in a porous medium, heat transfer or groundwater flow. It takes the form 
\begin{eqnarray*}
\partial_t u - \Delta(u^m) = F(t,u), \quad u(0) = u_0,
\end{eqnarray*}
where $m>0$, $F(t,u)$ is a source term, possibly also depending on $\nabla u$, and $u_0$  an initial value.  
For $m=1$ we recover the heat equation. 
The PME  is easily brought into the form \eqref{1.1} with $A(u) = - mu^{m-1}\Delta$. Maximal regularity for $A(u_0)$ can be derived from the bounded $H_\infty$-calculus of $c-\Delta$, $c>0$, via perturbation theory for $\cR$-sectoriality. 
We show how the existence of short time solutions can be inferred from the Cl\'ement-Li theorem. Moreover, we sketch the proof of the existence of global solutions. We follow here \cite{RS2} and \cite{RS3}. 
We focus on strictly positive initial values $u_0$, as this leads to maximal regularity solutions. For non-negative $u_0$ a concept of weak solutions has been suggested by Aronson and Peletier; it can also be applied in the conically degenerate situation. 

We then address the question, how the shape of the tip is reflected in the solution. Already in the results mentioned above, the geometry comes in, as the possible choice of the domain is intimately related to the first non-zero eigenvalue of the Laplacian on the cross-section. 
In order to see more of this phenomenon, we may consider a (strictly positive) initial value $u_0$ of the form 
$$u_0(x) = c_0+\cO(x^\infty),$$   
where $c_0$ is a positive constant and $x$ is the distance to the tip. We show that the only asymptotics that the solution $u$ can develop are governed by the spectrum of the Laplacian on the cross-section of the cone.  

In \cite{RSh1,RSh2} Roidos and Shao studied the fractional PME, where $-\underline\Delta$ is replaced by $(-\Delta)^\sigma$, $0<\sigma<1$, which requires new methods. 
Moreover,  Roidos \cite{RY} was able to find a solution to the Yamabe flow equation on conic manifolds using the techniques developed for the PME.      

%XXXXX. More.  XXXXX

The Cahn-Hilliard equation is a phase-field or diffuse interface equation. It models the separation of  phases in a binary mixture, for example a two-component alloy, but many other applications are encountered.
One finds the equation stated in various forms. We consider here the initial value problem 
\begin{eqnarray*}
\partial_t u +\Delta^2u +\Delta(u-u^3) =0, \quad u(0) = u_0.
\end{eqnarray*}
Here $u$ denotes the concentration difference of the two phases. The sets where $u=\pm1$ correspond to areas of pure phases.  

The Cahn-Hilliard equation is semi-linear  and therefore of a simpler type than the PME. It might be solved by other methods, but maximal regularity techniques provide an elegant tool to derive the existence of short time solutions. We first focus on the results  in \cite{RS4} and \cite{RS1} that establish the existence of short time solutions and show how they can be derived in a simpler way. The new feature here is the analysis of the bi-Laplacian.  

Lopes and Roidos in \cite{LR1} and \cite{LR2} were able to prove the long time existence of solutions in dimensions 2 and 3. Moreover, they could define a semi-flow on the solutions of average zero with a global attractor.

%%%%%%%%%%%%%%%%%%%%%%%%%%%

The structure of this article is as follows. In Section 2, we will introduce the basics: manifolds with conical singularities, cone Sobolev spaces, and conically degenerate differential operators. The principal example is the Laplace-Beltrami operator $\Delta$  with respect to a conically degenerate metric. Moreover, we will recall the two symbols associated to a cone differential operator. They define ellipticity and characterize the Fredholm property between Sobolev spaces, see also \cite{S24}. 

Section 3 is devoted to the study of the possible closed extensions of a cone differential operator $A$ as an unbounded operator  in a cone Sobolev space $\cH^{s,\gamma}_p(\B)$, $s,\gamma\in \R$, $1<p<\infty$. Here we will also encounter the so-called model cone operator $\widehat A$ associated with $A$. It will play an important role in the proof of the existence of a bounded $H_\infty$-calculus  for $\Delta$. Theorem \ref{T3.4} recalls the structure of the domains of the maximal extensions of $A$ and $\widehat A$. As a consequence one obtain the structure of the domains of {\em arbitrary} closed extensions of $A$ and $\widehat  A$ in Corollary \ref{3.5}.  As an example we present the domains of the maximal closed extensions of $\Delta$ and $\widehat \Delta$ that have been determined in \cite{RS5} following the analysis in \cite{SS2} and \cite{GKM2}.

In Section 4 we sketch the concepts of bounded $H_\infty$-calculus, bounded imaginary powers and $\cR$-sectoriality. Theorem \ref{HDelta} goes back to \cite{SS2}. It shows how a good choice of the domain guarantees the existence of a bounded $H_\infty$-calculus and thus maximal regularity. This is exemplified in the case of the Laplacian. 

The porous medium equation (including subsections on the fractional PME and on the Yamabe problem) is treated in  Section 5.  Section 6 is devoted to  the  initial value problem for the Cahn-Hilliard equation.\medskip

\noindent{\bf Relation to other work}
This article focuses on the work of  Nikolaos Roidos and the author\footnote{supported by the DFG Priority Program `Geometry at Infinity' under Grant SCHR 319/9-1} and  their collaborators P. Lopes, J. Seiler and Y. Shao. Two more articles of interest that are not treated here are \cite{Ro1} by Roidos on complex powers and \cite{RoSa} by Roidos and Savas-Halilaj on the curve shortening flow in the conic setting. They also use maximal regularity techniques. 

Apart from these articles, there is a   wealth of literature on evolution equations on singular manifolds, too much to cite. A few articles that are similar in spirit and contain many references to further work should be mentioned. 
In \cite{IMS} and   \cite{MRS15}, Isenberg, Mazzeo and Sesum and Mazzeo, Rubinstein and Sesum  considered  the Ricci flow on asymptotically conical surfaces and surfaces with conical singularities, respectively. 
Grieser, Held, Uecker and Vertman \cite{GHUV} considered the Cahn-Hilliard problem with interface parameter $\varepsilon>0$  for phase transitions on manifolds with conical singularities and studied the limit as $\varepsilon \to 0$ using $\Gamma$-convergence. 
Bahuaud and Vertman \cite{BV} proved long time existence of the Yamabe flow, and   Vertman \cite{Ver} established short time (under certain conditions also long time) existence of the Ricci-deTurck flow on manifolds with incomplete edges.
Chen  et al. considered semi-linear parabolic equations with cone, edge and corner degeneracy  and established results on global existence and finite time blow-up, see \cite{CL0,CL1,CL2}.
Grillo and Muratori with various collaborators studied in a series of papers the porous medium equation on complete noncompact Riemannian manifolds, e.g. hyperbolic space or Cartan-Hadamard manifolds. See e.g. \cite{GMP18}, or \cite{V15} for related work. Also the fractional porous medium equation has been studied in this setting, e.g. by Berchio, Bonforte,  Grillo and  Muratori in \cite{BBGM24}.

The cone calculus is closely related to Melrose's $b$-calculus \cite{Mel93}. At various places I have pointed out similarities. For a comparison of both approaches see also \cite{LauterSeiler}. 

\subsection*{Acknowledgement} I thank the referees for many valuable comments.

%%%%%%%%%%%%%%%%%%%%%%%%%%%%%%%%
\section{Conical Singularities and Cone Sobolev Spaces}
%%%%%%%%%%%%%%%%%%%%%%%%%%%%%%%%

%%%%%%%%%%%%%%%%%%%%%%%%%%%%%%%%
\subsection{Manifolds with Conical Singularities} \label{sect.2.1}
%%%%%%%%%%%%%%%%%%%%%%%%%%%%%%%%

%A good reference for this elementary section is \cite{S24}. 

%\begin{definition} 
An $(n+1)$-dimensional  manifold with conical singularities is a topological space $B$ that is a smooth manifold of dimension $n+1$  outside a finite subset $D\subseteq B$  of isolated points.
Moreover,  near every  $d\in D$, the space $B$ has the structure of a cone. More precisely,  there exist an open neighborhood $U_d\subseteq B$ of $d$ and a homeomorphism $\phi_d: U_d\to ({[0,1}[\times Y_d)/(\{0\}\times Y_d)$ that restricts to a diffeomorphism $\phi_d: U_d\setminus \{d\}\to {]0,1[}\times Y_d$. The space $Y_d$, the cross-section of the cone, is a smooth manifold of dimension $n$. 

If $\phi_d':   U_d\to ([0,1[\times Y_d)/(\{0\}\times Y_d)$ is another map with the above properties, we require that $\phi_d'\phi_d^{-1}: {]0,1[}\times Y_d\to {]0,1[}\times Y_d$ extends to a diffeomorphism   ${]-1,1[}\times Y_d\to {]-1,1[}\times Y_d$.
%\end{definition} 

This assumption allows us to blow up the singularities. We obtain an $(n+1)$-dimensional manifold $\B$ with boundary $\partial \B =\bigsqcup_{d\in D} Y_d $, which is analytically a much simpler object to work with.  In the sequel, we will assume $\B$ to be compact and fix the notation % for the dimension: 
\begin{eqnarray}
%\label{}
%\lefteqn{
\dim \B=n+1.
\end{eqnarray}

\begin{remark}
One can also define manifolds with boundary and conical singularities. In this case, $B\setminus D$ and the spaces $Y_d$ are manifolds with boundary of dimensions $n+1$ and $n$, respectively. 
\end{remark} 

In order  to model the conical singularities, we endow $\B$ with a degenerate Riemannian metric $g$. We assume $g$ to be a smooth Riemannian metric away from the boundary. Near the boundary, we fix a collar neighborhood ${[0,1[}\times \partial \B$  with coordinates $(x,y)$ and suppose that  $g$ is of the form 
\begin{eqnarray}
\label{eq.g}
%\lefteqn{
g = dx^2 + x^2h(x),
\end{eqnarray}
where ${[0,1[}\ni x\mapsto h(x)$ is a smooth family of (non-degenerate) Riemannian metrics on $\partial \B$. We will say that $\B$ has straight conical singularities, if $h$ is independent of $x$. 

\subsection{Cone Differential Operators}
\begin{definition}
A differential operator $A$ of order $\mu$ on $\B$ is called conically degenerate or a cone differential operator, if it is a differential operator of order $\mu$ with smooth coefficients on the interior $\Int \B$ of $\B$, and, near the boundary,  can be written in the form 
\begin{eqnarray}
\label{eq.conedeg}
%\lefteqn{
 A = x^{-\mu}\; \sum_{j=0}^\mu a_j(x)(-x\partial_x)^j,
\end{eqnarray}
where ${[0,1[}\ni x\mapsto a_j(x)$ is a smooth family of differential operators of order $\mu-j$ on $\partial \B$, $j=0,\ldots,\mu$. 
\end{definition} 

\begin{example} 
\label{Laplace}
The Laplace-Beltrami operator with respect to the metric $g$ is  
\begin{eqnarray}
\label{eq.Delta}
%\lefteqn{
\Delta = x^{-2} \left((-x\partial_x)^2 -(n-1 +H(x))(-x\partial_x) + \Delta_{h(x)}  \right)
\end{eqnarray}
and thus a cone differential operator of order $2$. Here, $\Delta_{h(x)}$ is the Laplace-Beltrami operator on $\partial \B$ with respect to the metric $h(x)$, and 
\begin{eqnarray}
\label{eq.H}
%\lefteqn{
H(x)= \frac x2 \frac{\partial_x(\det h(x))}{\det h(x)} = \mathcal O(x). 
\end{eqnarray}
\end{example} 

\begin{definition}
To a  cone differential operator $A$ as in \eqref{eq.conedeg} we associate two symbols:
\begin{enumerate}\renewcommand{\labelenumi}{(\roman{enumi})}
\item The (standard) principal pseudodifferential symbol $\sigma^\mu_\psi (A) $. Near $\partial \B$ it is  of the form $$\sigma^\mu_\psi (A)(x,y,\xi,\eta) = \sum_{j=0}^\mu \sigma_\psi^{\mu-j}(a_j(x))(y,\eta) (-ix\xi)^j.
$$  
with  the principal symbol $\sigma_\psi^{\mu-j}(a_j(x))$ of $a_j(x)$. 
Close to $\partial \B$ we also define the rescaled principal symbol: 
$$x^\mu \sigma_\psi^\mu(A)(x,y,\xi/x,\eta)
= \sum_{j=0}^\mu \sigma_\psi^{\mu-j}(a_j(x)) (y,\eta) (-i\xi)^j.$$   
\item The {\em conormal symbol} or {\em principal Mellin symbol}  of $A$  is the operator-valued polynomial $\sigma_M(A) :\mathbb C\to \scrL(H^\mu(\partial \B), L^2(\partial \B))$ defined by \footnote{This is essentially the {\em indicial family} in \cite[(5.7)]{Mel93}. While $z$ here corresponds to $-x\partial_x$, the variable $\lambda$ in \cite[(5.7)]{Mel93} corresponds to $xD_x$. The role played by the real part of $z$, e.g., in Definition \ref{2.15}, below, is therefore played by the imaginary part of $\lambda$ in \cite{Mel93}.}
$$\sigma_M(A)(z) = \sum_{j=0}^\mu a_j(0) z^j.$$
\end{enumerate} 
\end{definition}

\begin{example}For the Laplace-Beltrami operator with respect to the metric $g$ in \eqref{eq.g} we obtain near $\partial \B$:
\begin{eqnarray}
\label{2.6.1}
%\lefteqn{
\sigma_\psi^2(\Delta)(x,y,\xi,\eta)  = x^{-2} \left( (-ix\xi)^2 +\sigma_\psi^2(\Delta_{h(x)})\right);
\end{eqnarray}
the rescaled version is
\begin{eqnarray}
\label{2.6.2}
%\lefteqn{
x^2\sigma_\psi^2(\Delta)(x,y,\xi/x, \eta)  =  -\xi^2 - \|\eta\|^2_{h(x)},
\end{eqnarray}
where $\|\eta\|_{h(x)}$ is the norm of  $\eta\in T^*_y\partial \B$ induced by the metric $h(x)$. 

The principal Mellin symbol is 
\begin{eqnarray}
\label{2.6.3}
%\lefteqn{
\sigma^2_M(\Delta)(z) = z^2 -(n-1)z+\Delta_{h(0)}.
\end{eqnarray}
Note that $H$ does not appear, since $H(0)=0$.  
\end{example}

%See also \cite{S24} for a brief introduction to the analysis on manifolds with conical singularities.
%%%%%%%%%%%%%%%%%%%%%%%%%%%%%%%%
\subsection{Cone Sobolev Spaces}  
%%%%%%%%%%%%%%%%%%%%%%%%%%%%%%%%

The definition of cone differential operators suggests that the right choice of Sobolev spaces on manifolds with conical singularities should take into account the fact that derivatives with respect to $x$ always appear in the form $x\partial_x$ and that one  has powers of $x$ as weight factors. For $k\in \N_0$ and $\gamma\in \R$ we therefore define the space $\cH^{k,\gamma}_p(\B)$ by \footnote{For $p=2$, a corresponding family of spaces $\rho^a H^m_b(X;{{}^b\Omega}^\frac12)$ was 
introduced in \cite[(5.44)]{Mel93}.} 
\begin{eqnarray}
\label{eq.hsg}
\lefteqn{\cH^{k,\gamma}_p(\B) = \Big\{u\in H^k_{p, loc}(\Int \B): }\\ 
&&x^{\frac{n+1}2-\gamma}(x\partial_x)^jD_y^\alpha (\omega u) \in L^p(\B, \frac{dxdy}x)
\;\text{\rm for }j+|\alpha|\le k
\Big\} \nonumber.
\end{eqnarray}

Here $\omega$ is a {\em cut-off function near} $x=0$, i.e. $0\le \omega\in C^\infty_c([0,1[)$, $\omega (x) =1$ for small $x\ge 0$.  
This yields a norm on $\cH^{k,\gamma}_p(\B)$. Different choices for $\omega$ yield equivalent norms. 

For $p=2$,  $\cH^{0,0}_2(\B) = L^2(\B, \mu_g)$  is the $L^2$ space for $\B$ with the measure $\mu_g$ induced by the metric $g$, up to equivalence of norms (we will obtain the norm in $L^2(\B,\mu_g)$ if we choose the measure $\sqrt{\det h(x)}\frac{dxdy}x$ in \eqref{eq.hsg}). This motivates the factor $x^{(n+1)/2}$ in \eqref{eq.hsg}. In general, up to equivalent norms, 
$$L^p(\B,\mu_g) = \cH_p^{0,\gamma_p}(\B), \quad\text{with }  \gamma_p = (n+1)\left( 1/2 - 1/p\right),$$
by \cite[Remark 2.12(c)]{SS0}.
We next obtain $\cH^{s,\gamma}_p(\B)$ for $s\ge0$ by interpolation and finally for $s\le0$ by duality. There also is a direct way to introduce these spaces: Denote by $H^s_p(\R\times \partial\B)$ the standard $L^p$-Sobolev space on  $\R\times \partial\B$.  One then has the following alternative definition.

\begin{definition}\label{eq.hsg.2}For $s,\gamma\in \R$ and $1\le p < \infty$, $\cH^{s,\gamma}_p(\B)$ is the set of all $u\in H^{s}_{p,loc} (\Int \B)$ such that, for every cut-off function $\omega$ as defined after \eqref{eq.hsg},
$$ S_\gamma (\omega u) \in H^{s}_p(\R\times \partial\B ).$$
Here $S_\gamma$ is (the extension to distributions of) the map 
\begin{eqnarray*}
S_\gamma:  C^\infty_c(\R_+\times \partial\B) \to C_c^\infty (\R\times \partial\B), \quad 
(S_\gamma \varphi)(r,y) = e^{(\frac{n+1}2-\gamma)r} \varphi(e^{-r},y),
\end{eqnarray*}
and $\R_+ = {]0,\infty[}$. This makes $\cH^{s,\gamma}_p(\B)$ a Banach space. 
\end{definition}

That this definition actually furnishes the same space is easily checked by first considering the case  $s=0$ and then observing how derivatives transform under $S_\gamma$.  
 
\begin{corollary}\label{2.6}Let $1\le p<\infty$ and $s>(n+1)/p$, $\gamma\in \R$. Then a function   $u\in \cH^{s,\gamma}_p(\B)$  is continuous on $\Int \B$. Near $\partial \B$ it satisfies the estimate 
\begin{eqnarray*}
|u(x,y) |\le cx^{\gamma - \frac{n+1}2}\|u\|_{\cH^{s,\gamma}_p(\B)}
\end{eqnarray*}
for some constant $c\ge0$. 
In particular, $u(x,y) \to 0$ as $x\to 0$ whenever $\gamma>(n+1)/2$. 
\end{corollary}

\begin{proof}Since $u\in H^{s}_{p,loc} (\Int \B) $ for $s>(n+1)/p$, it is continuous in the interior of $\B$. 
The fact that $S_\gamma(\omega u) \in H^s_p(\R\times \partial\B)$ implies that 
\begin{eqnarray*}
|e^{(\gamma-\frac{n+1}2 )r} (\omega u)(e^{-r} ,y) |\le c_1 \|S_\gamma (\omega u)\|_{H^s_p(\R\times \partial\B)}\le c_2 \|\omega u \|_{\cH^{s,\gamma}_p(\B)} 
\end{eqnarray*}
for suitable constants $c_1$ and $c_2$. 
For $x=e^{-r}$ we obtain the  assertion. 
\end{proof}

\begin{lemma}\label{2.7}
In the situation of Corollary \ref{2.6}, we  have the weaker estimate
$u(x,y) \to 0$ as $x\to 0$ when $\gamma = (n+1)/2$.
\end{lemma} 

\begin{proof} We only need to consider $\omega u$ for a cut-off function $\omega$. As $S_{(n+1)/2}(\omega u) \in H^s_p(\R\times \partial\B)$, $(\omega u)(e^{-r},y) \to 0 $ as $r\to \infty$.
We then obtain the assertion for $x=e^{-r}$. 
\end{proof}

The following lemma is an immediate consequence of  \eqref{eq.hsg}.
% $\cH^{s,\gamma}_p(\B)$.  

\begin{lemma}\label{2.8} 
Let $\omega$ be a cut-off function near $x=0$ and $0\not=c\in C^\infty (\partial \B)$. 
For $q\in \R$, $k\in \N$,  let $u(x,y) = \omega(x) x^{-q} \ln^k x c(y)$. Then $u\in \cH^{\infty,\gamma}_p(\B)$  
if and only if  $q< (n+1)/2-\gamma$. 
\end{lemma}

\begin{remark} \label{2.9} 
Lemma \ref{2.6}, Corollary \ref{2.7} and Lemma \ref{2.8} show that  a given space $\cH^{s,\gamma}_p(\B)$, $s>(n+1)/p$, either contains functions that diverge as $x\to0$, or all functions in this space tend to zero as $x\to 0$.
\end{remark} 

We will need the following multiplier lemma: 

\begin{lemma}\label{multiplier}
Let $1<p,q<\infty$, $\sigma>0$, $\gamma\in \R$. Then multiplication by $m\in \cH^{\sigma+(n+1)/q,(n+1)/2}_q(\B)$ defines a bounded linear operator on $\cH^{s,\gamma}_p(\B)$ for $-\sigma<s<\sigma$. 
\end{lemma}

\begin{proof} The function $m$ belongs to $H^{\sigma+(n+1)/q}_{q,loc}(\Int\B)$ and thus locally to the Zygmund space $C^\sigma_*$. It therefore acts as a multiplication operator on $H^s_p$ for $-\sigma<s<\sigma$.  
In order to study the behavior near the boundary, we can assume $m$ to be supported in a single coordinate neighborhood. By definition, the function $m(e^{-r},y)$ then is an element of $H^{\sigma +(n+1)/q}_q(\R\times \partial \B)\hookrightarrow C^\sigma_*(\R\times \partial \B)$. It therefore acts as a bounded multiplier on $H^s_p(\R\times \partial \B)$ for $-\sigma<s<\sigma$.  
\end{proof} 

%
%{\color{magenta}
\begin{corollary}\label{algebra} $\cH^{s_0,\gamma_0}_p(\B)$ is a Banach algebra for  $s_0>(n+1)/p$ and $\gamma_0\ge(n+1)/2$ up to an equivalent norm.  
\end{corollary}
%}

%

\begin{theorem}\label{2.10}The cone Sobolev spaces have the following properties for $s,\gamma\in\R$ and $1<p<\infty$. 
\begin{enumerate}\renewcommand{\labelenumi}{{\rm(\alph{enumi})}}
\item $\cH^{s,\gamma}_p(\B)\hookrightarrow \cH^{s',\gamma'}_p(\B)$ for $s\ge s'$ and $\gamma\ge \gamma'$
\item The embedding in (a) is compact if, and only if, $s>s'$ and $\gamma>\gamma'$.
\item Complex interpolation furnishes 
$$\cH^{\tilde s+\varepsilon,\tilde \gamma+\varepsilon}_p(\B)\hookrightarrow[\cH^{s,\gamma}_p(\B), \cH^{s',\gamma'}_p(\B)]_\sigma\hookrightarrow \cH^{\tilde s-\varepsilon,\tilde \gamma-\varepsilon}_p(\B),$$
where $\tilde s = s+\sigma(s'-s)$, $\tilde\gamma= \gamma+\sigma(\gamma'-\gamma)$, and $\varepsilon >0$ is arbitrary. 
For $\gamma=\gamma'$,   
$$[\cH^{s,\gamma}_p(\B), \cH^{s',\gamma}_p(\B)]_\sigma= \cH^{\tilde s, \gamma}_p(\B).$$
\end{enumerate} 
\end{theorem} 

\begin{proof}
(a) is immediate from the definition; (b) follows from the corresponding property of the standard Sobolev spaces. The first part of (c) is \cite[Lemma 5.4]{CSS0}. For $\gamma=(n+1)/2$, the second follows from the standard interpolation theorem for Sobolev spaces and Definition \ref{eq.hsg.2}. For general $\gamma$ use the fact that multiplication by a smooth positive function that is equal to $x^\sigma$ near $\partial \B$ and equal to  $1$ outside a neighborhood of $\partial \B$ furnishes an isomorphism from $\cH^{s,\gamma}_p(\B)$ to $\cH^{s,\gamma+\sigma}_p(\B)$ for any $\gamma$. 
\end{proof}

There is another class of spaces that is important for the analysis on conic manifolds. 

\begin{definition}\label{2.11} 
For $s,\gamma\in\R$ and $1<p<\infty$,  $\cK^{s,\gamma}_p(\R_+\times\partial \B)$ consists of all distributions $u$ on $\R_+\times\partial \B$ such that for every cut-off function $\omega$
\begin{enumerate}\renewcommand\labelenumi{\rm(\roman{enumi})}
\item $\omega u \in \cH^{s,\gamma}_p(\B)$, and 
\item given a coordinate map $\kappa:U\subseteq\partial \B\to \R^n$ and $\phi\in C^\infty_c(U)$, the push forward $\chi_*((1-\omega)(x)\phi(y) u)$ is an element of $H^{s}_p(\R^{1+n})$, where $\chi(x,y) = (x,x\kappa(y))$. 
\end{enumerate}
$\cK^{s,\gamma}_p(\R_+\times\partial \B)$ can then be given the topology of a Banach space. 
\end{definition} 

Away from the boundary of $\B$, $\cK^{s,\gamma}_p(\R_+\times \partial \B)$ is the canonical Sobolev space $H^{s}_p$ on the far end of the cone with cross-section $\partial \B$, defined by considering $x\in (0,\infty)$ as a global coordinate. Alternatively, we may consider it there as the Sobolev space on $\R_+\times \partial \B$ with respect to a metric that is conical at infinity.\footnote{See \cite[Sect. 4.15]{Mel93} for a corresponding geometric model.} 
These spaces were introduced for $p=2$ in \cite[Section 2.1.1]{Schu}.

From the definition of the cone differential operators and \eqref{eq.hsg}  we immediately
obtain

\begin{theorem}\label{2.14}
Let $A$ be as in \eqref{eq.conedeg}, $s,\gamma\in \R$ and $1<p<\infty$. Then
$$A: \cH^{s+\mu,\gamma+\mu}_p(\B)\to \cH^{s,\gamma}_p(\B)$$ 
is a bounded operator. 
\end{theorem}  

\begin{definition}\label{2.15} 
A cone differential operator  $A$ as in \eqref{eq.conedeg} is called elliptic with respect to the line 
$\Gamma_{(n+1)/2-\gamma}= \{z\in \C: \re(z) = (n+1)/2-\gamma\}$, $\gamma\in \R$, if 
\begin{enumerate}\renewcommand{\labelenumi}{(\roman{enumi})}
\item the principal pseudodifferential symbol  $\sigma^\mu_\psi(A)$ is invertible on $T^*(\Int \B)\setminus 0$, and the rescaled symbol extends smoothly to $x=0$, remaining invertible up to $x=0$ for $(\xi,\eta)\not=0$.
\item The principal Mellin symbol $\sigma^\mu_M(A)(z)\in \scrL(H^\mu(\partial \B), L^2(\partial \B))$ is invertible for every $z\in \Gamma_{(n+1)/2-\gamma}$. 
\end{enumerate} 
\end{definition}

\begin{remark}\label{2.16} 
It seems arbitrary to assume  invertibility  in $\scrL(H^\mu(\partial \B),L^2(\partial \B))$. However, the spectral invariance of the differential operators implies that we might have equivalently required invertibility  in $\scrL(H^{s+\mu}_p(\partial \B), H^s_p(\partial \B))$ for any $s\in \R$ and $1<p<\infty$ (or many other spaces). 
\end{remark}

As usual, ellipticity implies the Fredholm property. 

\begin{theorem}\label{2.17}
Let $A$ be as in \eqref{eq.conedeg}, $s,\gamma\in \R$, $1<p<\infty$.  
Then 
$$A: \cH^{s+\mu,\gamma+\mu}_p(\B)\to \cH^{s,\gamma}_p(\B)$$
is a Fredholm operator, if, and only if, $A$ is elliptic with respect to the line $\Gamma_{(n+1)/2-\gamma -\mu}$ in the sense of Definition \ref{2.15}.
In this case, the index is independent of $s$ and $p$. The index does depend on $\gamma$, but it is locally constant near all $\gamma$ for which $\sigma_M^\mu(A)(z)$ is invertible for every  $z\in \Gamma_{(n+1)/2-\gamma}$.   
\end{theorem} 

This result goes back to Kondrat'ev \cite{K1967}. For a proof  of this version involving  $L_p$-Sobolev spaces see \cite[Corollaries 3.3 and 3.5, Theorem 3.13]{SS0}. \footnote{The corresponding result in the $b$-calculus with $p=2$ is \cite[Theorem 5.40]{Mel93}. In case $\sigma_\psi^\mu(A)$ satisfies (i) in Definition \ref{2.15}, but $\sigma^\mu_M(A)$  is not invertible on all of $\Gamma_{(n+1)/2-\gamma-\mu}$, Lesch \cite[Proposition 1.3.16]{Lesch}  showed that the minimal extension of $A$ in $\cH^{s,\gamma}_p(\B)$,  whose domain is larger than $\cH^{s+\mu,\gamma+2}_p(\B)$, is a Fredholm operator. See Section \ref{Sect.3.2}, below,  for the minimal extension.}

\begin{example}\label{2.18}
For the Laplace-Beltrami operator $\Delta$ we immediately see from Equations \eqref{2.6.1} and \eqref{2.6.2} that condition (i) in Definition \ref{2.15} is fulfilled. 

Concerning  (ii): The points of non-invertibility of the principal Mellin symbol
\begin{eqnarray}
\label{conormal}
\sigma_M(\Delta)(z) = z^2 -(n-1)z +\Delta_{h(0)}
\end{eqnarray}
are
\begin{eqnarray}
\label{qj}
q_{j}^{\pm}=\frac{n-1}{2}\pm\sqrt{\Big(\frac{n-1}{2}\Big)^{2}-\lambda_{j}},\ j=0,1,2, \ldots,
\end{eqnarray}
where $0=\lambda_0>\lambda_1> \lambda_2>\ldots$ are the different eigenvalues of $\Delta_{h(0)}$.
Hence $\Delta$ is elliptic with respect to any weight line that does not contain any of the $q_j^\pm$, and, given $s\in \R$, $1<p<\infty$, 
$$\Delta: \cH^{s+2,\gamma+2}_p(\B) \to \cH^{s,\gamma}_p(\B)$$ 
is a Fredholm operator if, and only if, $(n+1)/2-\gamma-2 \not= q_j^\pm$ for all $j$. 
  
Denote by $E_j$ the $\lambda_j$-eigenspace of the Laplacian $\Delta_{h(0)}$ on $\partial \B$ and by $\pi_j$  the orthogonal projection in $L^2(\partial \B)$ onto  $E_j$. Note that $\lambda_0 =0$ and $E_0$ consists of locally constant functions on $\partial \B$. We moreover see that
\begin{eqnarray}\label{inverse}
\sigma_M(\Delta)^{-1}(z) = \sum_{j=0}^\infty \frac{\pi_j}{(z-q_j^+)(z-q_j^-)}.
\end{eqnarray}
Hence the poles of $z\mapsto \sigma_M(z)^{-1}$ are all simple unless $n=1$, when $z=q_0^+=q_0^-=0$ is a double pole. 
\end{example} 

%%%%%%%%%%%%%%%%%%%%%%%%%%%%%%%%
\section{Closed Extensions of Cone Differential Operators}
%%%%%%%%%%%%%%%%%%%%%%%%%%%%%%%%

%%%%%%%%%%%%%%%%%%%%%%%%%%%%%%%%
\subsection{The Model Cone Operator} 
%%%%%%%%%%%%%%%%%%%%%%%%%%%%%%%%

To a cone differential operator $A=x^{-\mu} \sum_{j=0}^\mu a_j(x)(-x\partial_x)^j$ as in \eqref{eq.conedeg} we  associate the model cone operator 
\begin{eqnarray}
\label{eq.3.0}
%\lefteqn{
\widehat A =x^{-\mu}  \sum_{j=0}^\mu a_j(0) (-x\partial_x)^j.
\end{eqnarray}
It reflects the behavior of $A$ near the tip of the cone.\footnote{This is essentially the indicial operator in the $b$-calculus, see \cite[(4.102)]{Mel93}.} 
%We consider $\widehat A$ as an operator on $\R_+\times Y$. 

\begin{lemma}\label{3.1}  Let $s,\gamma\in \R$ and $1<p<\infty$. 
Then $\widehat A$ in \eqref{eq.3.0} defines a bounded linear map 
$$\widehat A: \cK^{s+\mu,\gamma+\mu}_p(\R_+\times \partial \B) \to \cK^{s,\gamma}_p(\R_+\times \partial \B).$$
\end{lemma}

Idea of the {\em proof}. Near $x=0$, the spaces $\cH^{s,\gamma}_p$ and $\cK^{s,\gamma}_p$ coincide. Thus the assertion for this part follows from Theorem \ref{2.14}. Away from $x=0$, the map $\chi$ introduced in Definition \ref{2.11} leads to a transformation of the cylinder $\R_+\times \partial \B$  to an outgoing (opening) cone with cross-section $\partial \B$, and the derivatives $\partial_{y_j}$, in local coordinates on $\partial \B$, transform to  $x\partial_{y_j}$. Hence  $\chi_*\widehat A$
%$$\chi_*\Big(x^{-\mu} \sum_{j=0}^\mu a_j(0)(-x\partial_x)^j\Big) = \sum_{j=0}^\mu a_j(0)(-\partial_x)^j.$$  
is a differential operator of order $\mu$ with coefficients  bounded in all derivatives. It therefore yields a bounded map $H_p^{s+\mu}\to H_p^s$ for the Sobolev spaces on the outgoing cone. 
\hfill$\Box$

\begin{example}The model cone operator of the Laplace-Beltrami operator $\Delta$  is 
$$\widehat \Delta = x^{-2} \big((-x\partial_x)-(n-1)(-x\partial_x) +\Delta_{h(0)}\big),$$ 
where $\Delta_{h(0)}$ is the Laplace-Beltrami operator on $\partial \B$ given by the metric $h(0)$. 
\end{example} 

\subsection{Closed Extensions}\label{Sect.3.2}
An elliptic pseudodifferential operator of order $\mu$ on a closed manifold $M$, considered as an unbounded operator in $H^s_p(M)$, $s\in\R$, $1<p<\infty$,  has precisely one closed extension, namely that with domain $H^{s+\mu}_p(M)$. 
For cone differential operators the situation is quite different. In the sequel, we will denote by $\scrD(B)$ the domain of a given operator $B$. 

\begin{remark}\label{R3.3} 
{\bf (Minimal and maximal extension).}
 For  $s,\gamma\in \R$ and $1<p<\infty$ consider the cone differential operator  $A=x^{-\mu}  \sum_{j=0}^\mu a_j(x) (-x\partial_x)^j $ as an unbounded operator in the space $\cH^{s,\gamma}_p(\B)$ with initial domain $C^\infty_c(\Int \B)$. Then $A$ has two canonical closed extensions, namely 
the {\em minimal extension}, i.e., the closure of $A$, and the {\em maximal extension}, whose domain consists of all $u\in \cH^{s,\gamma}_p(\B)$ for which $Au\in \cH^{s,\gamma}_p(\B)$.   
We denote these extensions  by $A_{\min}$ and $A_{\max}$, respectively.  A similar statement holds for $\widehat A$ in $\cK^{s,\gamma}_p(\R_+\times \partial \B)$. Our next goal is to understand these two extensions. 
\end{remark}

\begin{theorem}\label{T3.4}
Let $A$ be a cone differential operator of order $\mu$ that is elliptic with respect to the line  $\Gamma_{(n+1)/2-\gamma-\mu}$ in the sense of Definition \ref{2.15}. Consider $A$ as an unbounded operator in $\cH^{s,\gamma}_p(\B)$ for some choice of $s, \gamma\in \R$ and $1<p<\infty$. Then

%\begin{enumerate}\renewcommand{\labelenumi}{(\alph{enumi})}%
%\item 
\noindent{\rm (a)} $%\begin{eqnarray*}
\scrD(A_{\min}) = \cH^{s+\mu,\gamma+\mu}_p(\B)\quad\text{and }\quad
\scrD(\widehat A_{\min}) = \cK^{s+\mu,\gamma+\mu}_p(\R_+\times \partial \B).
$%\end{eqnarray*}

%\item 
\noindent{\rm (b)} If $A$ only satisfies ellipticity condition (i) in Definition \ref{2.15}, the minimal domains are
\begin{eqnarray*}
\scrD(A_{\min} ) &=& \left\{ u\in \bigcap_{\varepsilon>0} \cH^{s+\mu, \gamma+\mu-\varepsilon}_p(\B) :
Au\in \cH^{s,\gamma}_p(\B)\right\}\, \text{ and} \\
\scrD(\widehat A_{\min} ) &=& \left\{ u\in \bigcap_{\varepsilon>0} \cK^{s+\mu, \gamma+\mu-\varepsilon}_p(\R_+\times \partial \B) : \widehat Au\in \cK^{s,\gamma}_p(\R_+\times \partial \B)\right\}.
\end{eqnarray*}

%\item 
\noindent{\rm (c)} In both cases, the maximal domains are 
\begin{eqnarray*}
\scrD(A_{\max}) = \scrD(A_{\min}) \oplus \scrE\quad \text{and } \quad
\scrD(\widehat A_{\max}) = \scrD(\widehat A_{\min}) \oplus \widehat \scrE,
\end{eqnarray*}
where $\scrE$ and $\widehat \scrE$ are finite-dimensional spaces of smooth functions on $\Int \B$ with special asymptotics near $\partial \B$ that can be determined explicitly. 
There is a canonical isomorphism 
\begin{eqnarray}
\label{theta}
%\lefteqn{
\Theta: \scrD(A_{\max})\to \scrD(\widehat A_{\max}).
\end{eqnarray}
\end{theorem} 

For $A_{\min}$ and $A_{\max}$ this theorem goes back to Br\"uning and Seeley \cite{BS88} and 
Lesch \cite[Section I.1.3]{Lesch}. It has been extended by Gil, Krainer and Mendoza, see e.g.,  \cite{GM1} or \cite[Theorem 4.7]{GKM2}, where also the isomorphism $\Theta$ was constructed. 
A particularly simple way to compute the spaces $\scrE$ and $\widehat \scrE$ was presented in \cite[Theorem 3.5]{SS2}.

The spaces $\scrE$, $\widehat\scrE$ and the isomorphism $\Theta$ are given as follows: Let $\scrM: C^\infty_c(\R_+\times \partial \B)\to \scrH(\C,C^\infty( \partial \B))$ be the Mellin transform ($\scrH$ denotes holomorphic functions), defined by\footnote{In \cite[(5.1)]{Mel93} the Mellin transform is defined by $\scrM u(\lambda,y) = \int_0^\infty x^{i\lambda-1}u(x,y) dx$. } 
$$\scrM u(z,y) = \int_0^\infty x^{z-1} u(x,y) \,dx.$$ 
According to \cite[Theorem 3.4]{SS2}, the space $\widehat\scrE$ is a direct sum of spaces
$$\widehat \scrE = \bigoplus \widehat \scrE_\sigma,$$
with  $\sigma$ running  over the poles of $\sigma_M(A)^{-1}$ in the interval 
$$J_\gamma= {](n+1)/2-\gamma-\mu, (n+1)/2-\gamma[}\,.$$ 
Moreover,  
$$\widehat \scrE_\sigma =\omega\, \text{\rm range} \, G_\sigma$$
for a  cut-off function  $\omega$ and  $G_\sigma : C^\infty_c(\R_+\times \partial \B)\to C^\infty(\R_+\times \partial \B)$  given by 
$$(G_\sigma u) (x,y) = \frac1{2\pi i} \int_{|z-\sigma|=\varepsilon} t^{-z} \sigma_M(A)^{-1} (z) \scrM u(z,y) \, dz
$$
for sufficiently small $\varepsilon >0$. The functions in $\widehat \scrE_\sigma$ then are  of the form $\omega(x)\, x^{-\sigma} \ln^k\!x \, e(y)$. Here, $k\in \{0,1,\ldots, m-1\}$, where $m$ is the order of the pole in $\sigma$, and $e$ are certain smooth functions on $\partial \B$. 
Similarly, the space $\scrE$ is a direct sum 
$$ \scrE= \bigoplus_{\sigma\in J_\gamma} \scrE_\sigma.$$
Now, however, 
\begin{eqnarray}
\label{Gsigma1}
%\lefteqn{
\scrE_\sigma = \omega\, \text{\rm range}\,(G_\sigma + \widetilde G_\sigma ),
\end{eqnarray}
where the range of $\widetilde G_\sigma$ consists of functions that are $o(x^{-\sigma})$, see 
\cite[Theorem 3.5]{SS2} for the precise form. 
The map $\Theta$  in \eqref{theta} then is given by omitting the contribution of $\widetilde G_\sigma$. 

This will become clearer, when we see the maximal domain of $\Delta$ as an example.

\begin{corollary}\label{3.5} Let $A$ be elliptic of order $\mu$ with respect to the line $\Gamma_{(n+1)/2-\gamma -\mu}$. The domain of any closed extension $\underline A$  of $A$, considered as an unbounded operator in $\cH^{s,\gamma}_p(\B)$, is of the form 
\begin{eqnarray*}
\scrD(\underline A)&=& \cH^{s+\mu,\gamma+\mu}_p(\B) \oplus \underline\scrE,
\end{eqnarray*}
where $\underline\scrE$ is a subspace of $\scrE$. Conversely, any such domain defines a closed extension. 

Similarly,  the domain of any closed extension $\underline{\widehat A}$ of $\widehat A$, considered as an unbounded operator in $\cK^{s,\gamma}_p(\R_+\times \partial \B)$ is of the form 
\begin{eqnarray*}
\scrD(\underline{\widehat A}) &=& \cK^{s+\mu,\gamma+\mu}_p(\R_+\times\partial \B )\oplus \underline{\widehat \scrE}
\end{eqnarray*}
for a subspace $\underline{\widehat \scrE}$ of $\widehat \scrE$, and any such  domain defines a closed extension. 
\end{corollary}  

%\begin{remark}\label{3.6}
%In applications it is often important to find solutions to evolution equations on $\B$ that are bounded and nowhere vanishing. In Remark \ref{2.11} we saw that in any space $\cH^{s,\gamma}_p(\B)$, $s>(n+1)/p$, either all functions tend to zero as $x\to0$, or it contains elements that blow up as $x\to 0$. For this reason we will consider domains that contain functions that are constant near $\partial \B$. 
%%It is therefore important to be able to use domains that contain non-zero constant functions but no elements that blow up as $x\to 0$. 

%We will next explore the situation for the Laplace-Beltrami operator in $\cH^{s,\gamma}_p(\B)$.  Theorem \ref{3.7}  will provide a precise description of a maximal extension. This will lead to a domain with the desired properties in \eqref{4.7}.
%\end{remark} 

\begin{example}\label{3.6}
Consider the Laplace-Beltrami operator $\Delta$  associated with the metric $g$ in \eqref{eq.g} and the corresponding model cone operator $\widehat \Delta$ as unbounded operators in $\cH^{s,\gamma}_p(\B)$ and $\cK^{s,\gamma}_p(\R_+\times \partial\B)$, respectively. 
We introduce the interval
\begin{eqnarray}
\label{3.6.1}
I_\gamma = \left] \frac{n+1}2-\gamma-2, \frac{n+1}2-\gamma \right[.
\end{eqnarray}
Given the $q_j^\pm$ and the spaces $E_j$ from Example \ref{2.18}  the function spaces 
$\widehat \scrE_{q_j^\pm}$  are as follows: For $q_j^\pm$, $j\not=0$ or $n>1$ (i.e., whenever $q_j^\pm$ is a simple pole of $\sigma_M(\Delta)^{-1}$), let
 \begin{eqnarray}\label{3.3}
%\scrE_{q_j^\pm} &=& \left\{ u\in C^\infty(\Int \B): u(x,y) = x^{-q_j^\pm}\omega(x) e_j(y), e_j\in E_j , \text{ near } \partial \B\right\}\\
\widehat\scrE_{q_j^\pm} &=& \{ u\in C^\infty(\R_+\times \partial\B): 
u(x,y) = x^{-q_j^\pm}\omega(x) e_j(y), e_j\in E_j \},
\end{eqnarray}
 while, for $j=0$ and $n=1$, i.e., when $q_0^+=q_0^-=0$ is a double pole, 
 \begin{eqnarray}\label{3.4}
\lefteqn{
\widehat \scrE_0 %= \widehat\scrE_{q_0^\pm} 
= \{ u\in C^\infty(\R_+\times \partial\B): }\\
&&u(x,y) = \omega(x) (e_0(y) +e_0'(y) \ln x); e_0,e_0'\in E_0 \}.\nonumber
\end{eqnarray}
Then 
\begin{eqnarray}
%\label{}
%\lefteqn{
\widehat \scrE = \bigoplus_{q_j^\pm\in I_\gamma} \widehat\scrE_{q_j^\pm}.
\end{eqnarray}
If $\B$ has straight conical singularities, i.e., the metric $h$ in \eqref{eq.g} is independent of $x$, 
\begin{eqnarray}
\label{eq.3.6}
%\lefteqn{
 \scrE = \widehat \scrE,
\end{eqnarray}
where we  consider the functions $x^{-q_j^\pm}\omega(x) e_j(y)$, $q_j^\pm\in I_\gamma$,  and possibly also $\omega(x)e_0'(y)\ln x$  as elements of $C^\infty(\Int \B)$, supported in the collar neighborhood of
$\partial \B$. % the boundary. 
When $h$ depends on $x$, the structure of $\scrE$ may become more complicated, as we shall see in  Theorem \ref{3.7}, below.

We will next determine the maximal extension of $\Delta$ in $\cH^{s,\gamma}_p(\B)$ for a special choice of $\gamma$ that will later allow us to choose extensions of $\Delta$ that are particularly suitable for the applications we have in mind. Choose $k$ such that  the $q_j$ in \eqref{qj} satisfy 
$$ q^-_{k+1}\le -2 < q_k^-< \ldots <q_1^-<q_0^-=0\le q_0^+ <q_1^+<\ldots. $$
We then fix $\gamma$ with
\begin{eqnarray}\label{gamma1}
-2<\frac{n+1}2-\gamma-2 < q_k^-\ \text{ and }\ \ \frac{n+1}2-\gamma \not= q_j^+, \ j=0,1,\ldots .
\end{eqnarray}

The reason for this choice is the following: The interval $I_\gamma$ has length $2$ and will therefore contain $q_0^-=0$. Hence the maximal domain will contain $\scrE_{q_0^-}$ and therefore functions that are constant near $\partial \B$. We want these to belong to the domain of the extension we plan to work with. Moreover, $I_\gamma$ will contain as many of the $\scrE_{q_j^-}$, $j\ge 1$,  as possible. The functions in these spaces will show the asymptotics of the solutions near the conic points. We thus also want them to belong to the extension. Finally, the maximal domain can also contain  elements in $\scrE_{q_j^+}$ for $q_j^+>0$ that may blow up as $x\to 0^+$ and therefore should not belong to the domain of the extension.

%The reason for this choice is the following: The interval $I_\gamma$ has length $2$. We want it to include $q_0^-=0$, since $\scrE_{q_0^-}$  contains the constant functions, and, moreover, as many of the $\scrE_{q_j^-}$  as possible, as these will show the asymptotics of the solutions near the conic points.
%The non-zero elements in $\scrE_{q_j^+}$, $q_j^+>0$ are unbounded and therefore should not be in the domain.  
%%In order to see all $\scrE_{q_j^-}$, $j=0, \ldots, k$, we 
%We therefore fix $\gamma$ such that 
%\begin{eqnarray}\label{gamma1}
%-2<\frac{n+1}2-\gamma-2 < q_k^-%< \ldots <q_0^-=0. %<\frac{n+1}2-\gamma < q_0^+.
%\end{eqnarray}
%and $(n+1)/2-\gamma \not= q_j^+$ for $j=0, 1,\ldots$. 
%In case $q_0^+ = q_0^-$ ask instead that 
%\begin{eqnarray}\label{gamma2}
%{\color{white} mm} -2<\frac{n+1}2-\gamma-2 < q_k^-< \ldots <q_0^-=q_0^+=0<\frac{n+1}2-\gamma< q_1^+.
%\end{eqnarray}
%For optimal precision $(n+1)/2-\gamma-2$ should be close to $-2$, in particular, we will assume that $(n+1)/2-\gamma-2<-1$. 
%In order to state the result on the maximal domain i
In order to determine $\scrE_{q_j^-}$ for $j=0,\ldots, k$, introduce the operator-valued  function 
\begin{eqnarray}\label{eq.3.7}
b(z) =
- \sigma_M(\Delta)^{-1}(z-1)\left( (-\partial_x H)|_{x=0} z + (\partial_x\Delta_{h(x)})|_{x=0}\right)
%(\Pi_{q_j^-}(\sigma_M(A)^{-1}z \scrM u(z)) \, dz.  
\end{eqnarray}
with $H$ defined in \eqref{eq.H}.  As $0$ is the only possible double pole of $\sigma_M(A)^{-1}$, we can write $b$ near  $z=q_j^-$, $j\ge 1$,   in the form 
\begin{eqnarray}\label{eq.3.8}
b(z) \equiv b^{(-1)}_{q_j^-}(z-q_j^-)^{-1} +  b^{(0)}_{q_j^-}
\end{eqnarray}
modulo functions that are holomorphic near $q_j^-$ and vanish in $q_j^-$. 
This leads to: % the following result: 
\end{example}

\begin{theorem}\label{3.7}%With the choice of $\gamma$ in \eqref{gamma1}, %} and \eqref{gamma2}, respectively,  we c
Consider $\Delta$ as an unbounded operator in $\cH^{s,\gamma}_p(\B)$ for arbitrary $s\in\R$, $1<p<\infty$, and $\gamma$ as in \eqref{gamma1}. Then
\begin{eqnarray*} 
\scrD(\Delta_{\max}) = \cH^{s+2,\gamma+2}_p(\B) \oplus \bigoplus_{q^\pm_j\in I_\gamma}\scrE_{q_j^\pm}.
\end{eqnarray*} 
The spaces $\scrE_{q_j^-} $, $q_j^-\in I_\gamma$,  are determined from \eqref{eq.3.8} and the spaces $E_j$ defined in Example \ref{2.18}: \footnote{The spaces $\scrE_{q_j^+}$ can be determined in the same way, but we will not need them. }

%\begin{enumerate}\renewcommand{\labelenumi}{(\roman{enumi})}  
\noindent{\rm(i)} %Let $j\not=0$. Then 
If $q_j^- $ is a simple pole of $\sigma_M(\Delta)^{-1} $, then 
%\begin{itemize}
%\item If $q_j^--1<(n+1)/2-\gamma-2$ {\rm(}this is in particular the case when $q_j^-\le -1$, since $(n+1)/2-\gamma-2 >-2$ by assumption{\rm)}, then
%\begin{eqnarray*} 
%\scrE_{q_j^-}&=&  \{ u\in C^\infty(\Int \B): u(x,y) = \omega(x) x^{-q_j^-} e(y); e\in E_j\},
%\end{eqnarray*} 
%\item If $q_j^--1\ge (n+1)/2-\gamma-2$, then
\begin{eqnarray*} 
%\lefteqn{
\scrE_{q_j^-} &=&  \big\{ u\in C^\infty(\Int \B)\colon \\&& u(x,y) = \omega(x) \big(x^{-q_j^-} 
%}\\
%&&\ \ \ \ +
+x^{-q_j^-+1}\big(b^{(0)}_{q_j^-} +  (\ln x)\, b_{q_j^-}^{(-1)}\big)\big) e(y)  ; e \!\in E_j \big\}.
\end{eqnarray*}
%\end{itemize} 

%\noindent{\rm(ii)} Let $j=0$ and $n\ge2$. Then $q_0^-$ is a simple pole and
%$\scrE_0 =\widehat\scrE_0$.
%\begin{itemize}
%\item
%if $-1$ also is a pole of $(f_0)^{-1}=\sigma_M(\Delta)^{-1}$, say, $-1=q_\ell^-$, then 
%\begin{eqnarray*}
%\lefteqn{\scrE_0 = \{u\in C^\infty(\Int(\B)): }\\
%&&u(x,y) = \omega(x)(e_0+ x\pi_\ell((\partial_{x}H)_{|x=0}e_0)); e_0\in E_0\}. 
%\end{eqnarray*}
%\item if $-1$ is not a pole of $\sigma_M(\Delta)^{-1}$, then $\scrE_0 =\widehat \scrE_0$
%\end{itemize} 

\noindent{\rm(ii)} If  $j=0$ and  $n=1$, i.e., $0=q_0^-=q_0^+$ is a double pole of $\sigma_M(\Delta)^{-1}$: 
\begin{itemize} 
\item In case $-1$ also is  a pole of $\sigma_M(\Delta)^{-1}$, say for $q_\ell^-=-1$, we can write, 
modulo functions that are holomorphic near $z=0$ and vanish for $z=0$,
\begin{eqnarray*}
\sigma_M(\Delta)^{-1} (z-1) \equiv -\frac12\frac{\pi_\ell}z+S_0
\end{eqnarray*}
with $\pi_\ell$ from \eqref{inverse} and a suitable operator $S_0$. Then
\begin{eqnarray}\label{scrE0}
%\lefteqn{
&&\scrE_0 = \Big\{u\in C^\infty(\Int(\B)):\\
\nonumber
&&u(x,y) =\omega(x) \Big(e_0(y)+ e_0'(y)\ln x +\frac x2\pi_\ell((\partial_xH)|_{x=0}) e_0)(y)\\ 
\nonumber
&&-\frac{ x}2\left( (\ln x) \pi_\ell ((\partial_xH)|_{x=0} e_0')(y) -S_0((\partial_xH)|_{x=0} e_0')(y)\right)\Big) ; \, e_0,e_0'\in E_0\Big\}.
\end{eqnarray}

\item In case  $-1$ is not a pole of $\sigma_M(\Delta)^{-1}$,  
\begin{eqnarray*}
\scrE_0 &=& \big\{u\in C^\infty(\Int(\B)): u(x,y) = \omega(x)\big(e_0(y) +e_0'(y)\ln x \\
&&-x(1+\Delta_{h(0)})^{-1} ((\partial_xH)|_{x=0} e_0')(y)\big); \, e_0,e_0'\in E_0\big\}.
\end{eqnarray*}
\end{itemize}
%\end{enumerate} 
Here we identify a collar neighborhood of $\partial \B$ with  $\overline \R_+\times \partial \B$.  
\end{theorem} 

\begin{proof}
See the appendix of  \cite{RS5}.
\end{proof} 

\begin{remark}\label{3.8} 
{\rm (a)} By \eqref{eq.3.6},  $\scrE =\widehat {\scrE} $, if $h$ in \eqref{eq.g} is independent of $x$. The above result in connection with \eqref{eq.3.7} shows that it suffices that $\partial_x H(x)|_{x=0}=0$ for this to hold.

(b) Since, for arbitrary $s$, $\gamma$, and $1<p<\infty$, $x^{-q}(\ln^kx) c(y) \in \cH^{s+2,\gamma+2}_p(\B)$ for $q<(n+1)/2-\gamma-2$, $k\in \N_0$ and $c\in C^\infty(\partial \B)$, we may omit in Theorem \ref{3.7} all terms $x^{-q_j^-+1}(\ln^kx)e_j$ where $q_j^-<(n+1)/2-\gamma-1$. In particular, $\scrE=\widehat \scrE$ when $(n+1)/2-\gamma-2>-1$. 

(c) At first glance, it seems that $\scrE_0$ in \eqref{scrE0} might not contain the functions that are constant near $\partial\B$. A closer look, however,  shows that, due to the projection $\pi_\ell$, the term  $\frac x2\pi_\ell((\partial_xH|_{x=0}e_0)$ actually also belongs to $\scrE_{q_\ell^-}$  
\end{remark} 

%\begin{remark}\label{3.9} 
%Using the notation in Theorem \ref{3.7} the isomorphism $\Theta$ maps the spaces $\scrE_{q_j^-}

%%%%%%%%%%%%%%%%%%%%%%%%%%%%%%%%
\section{Maximal Regularity, the Cl\'ement-Li Theorem and Applications}
%Bounded $H_\infty$-Calculus, Bounded Imaginary Powers and $\cR$-Sectoriality}
%%%%%%%%%%%%%%%%%%%%%%%%%%%%%%%%
 
Let $\X_0,\X_1$ be Banach spaces, $\X_1$ densely embedded in $\X_0$. We assume them to be UMD spaces, see \cite[Section 4.2]{HVW} for a definition and many examples. 
All the spaces $\cH^{s,\gamma}_p(\B)$ and $\cK^{s,\gamma}_p(\R_+\times \partial \B)$, for $s,\gamma\in \R$, $1<p<\infty$, and their closed subspaces are UMD spaces.

\begin{definition}\label{MaxReg} Let $1<q<\infty$, $T>0$, and let $-B\in \scrL(\X_1,\X_0)$ be the infinitesimal generator of an analytic semigroup in $\X_0$ with domain $\scrD(B)=\X_1$.  Then $B$ is said to have maximal $L^q$ regularity, provided that the initial value problem in $L^q(0,T;\X_0)$
\begin{eqnarray}\label{MR}
\partial_t u + Bu =f, \quad u(0) = u_0,
\end{eqnarray}
with data $f\in L^q(0,T; \X_0)$ and $u_0$ in the real interpolation space $(\X_0,\X_1)_{1-1/q,q}$
has a unique solution
$$u\in W^{1}_q(0,T; \X_0)\cap L^q(0,T;\X_1)$$
depending continuously on the data.  
\end{definition} 

%Let $\X_1\hookrightarrow \X_0$ be Banach spaces 
%and  $A:\scrD(A)=\X_1 \to \X_0$  a closed unbounded linear operator. 

The importance of maximal regularity is illustrated by the following theorem, due to Cl\'ement and Li,
 \cite[Theorem 2.1]{CL}; see also \cite{P02} for a slightly more general version:

\begin{theorem} \label{ThmCL} In $L^{q}(0,T_{0};\X_{0})$,  $1<q<\infty$, $T_0>0$, consider the quasilinear  problem 
\begin{gather}\label{QL}
\partial_tu(t)+A(u(t))u(t)=f(t,u(t))+g(t), \;\; t\in(0,T_{0}),\quad u(0)=u_{0},	
\end{gather}
for $u_0$ in the real interpolation space $(\X_0,\X_1)_{1-1/q,q}$.
 
Assume that there exists an open neighborhood $U$ of $u_0$ in  $(\X_0,\X_1)_{1-1/q,q}$ such that $A(u_0): \X_{1}\rightarrow \X_{0}$ has maximal $L^{q}$-regularity and \\
{\rm (H1)} $A\in C^{1-}(U, \scrL(\X_1,\X_0))$, \\
{\rm (H2)}  $f\in C^{1-,1-}([0,T_0]\times U, \X_0)$,\\
{\rm (H3)} $g\in L^q(0,T_0; \X_0)$,\\
where `$C^{1-}$' denotes Lipschitz continuous functions.  Then there exist a $T>0$ and a unique solution $u\in W^1_q(0,T;\X_0) \cap L^q(0,T;\X_1)$ to  \eqref{QL} on $(0,T)$.
\end{theorem} 

The solution $u$ then belongs to $C([0,T]; (\X_0,\X_1)_{1-1/q,q})$  by \cite[Theorem III.4.10.2]{Amann}.\medskip

Instead of verifying the conditions in Definition \ref{MaxReg} it is often easier to establish stronger properties. 
We briefly recall  the concepts of bounded $H_\infty$-calculus, bounded imaginary powers, $\cR$-sectoriality and their relation to maximal regularity. For many more details see \cite{DHP} or \cite{KW04}.

Let $A:\scrD(A)=\X_1 \to \X_0$ be a closed unbounded linear operator. The sector 
\begin{eqnarray}\label{mingrowth}
\Lambda_\theta= \{z=re^{i\varphi}\in \C: r>0, |\varphi|\ge \theta\}
\end{eqnarray}
is said to be a sector of minimal growth for the resolvent of $A$, if $\lambda-A$ is invertible for $\lambda\in \Lambda_\theta$ and there exists a $C\ge0$ such that
\begin{eqnarray}\label{sectorial}
\|\lambda(\lambda-A)^{-1}\|_{\scrL(\X_0)}\le C, \quad \lambda\in \Lambda_\theta. 
\end{eqnarray}
Under this assumption one can define the operator $f(A)$ for any bounded holomorphic function $f:\C\setminus \Lambda_{\theta}\to \C$ (for short: $f\in H_\infty(\C\setminus \Lambda_\theta)$) by 
\begin{eqnarray*}
f(A) = \frac1{2\pi i}\int_{\partial\Lambda_{\theta}} f(\lambda) (\lambda-A)^{-1}\, d\lambda.
\end{eqnarray*}
The operator $A$ is said to admit a bounded $H_\infty$-calculus of angle $\theta$, if there exists a constant $C$, independent of $f$, such that
\begin{eqnarray*}
\|f(A)\|_{\scrL(\X_0)} \le C \|f\|_\infty, \quad f\in H_\infty(\C\setminus \Lambda_\theta).
\end{eqnarray*}
For a selfadjoint operator $A\ge 0$ on a Hilbert space $\X_0$ it follows from the functional calculus that the above estimate holds for $A+c$, $c>0$, and every bounded Borel function $f$ on the spectrum of $A$. 
In particular, if $\Delta$ is the Laplace-Beltrami operator on a closed Riemannian manifold $M$, then $c-\Delta$, $c>0$, has a bounded $H_\infty$-calculus on $\X_0=L^2(M)$.
For a Banach $\X_0$ the estimate is much more subtle. 

The notion of bounded $H_\infty$-calculus goes back to Alan McIntosh's seminal paper \cite{McI86}. 
A slightly weaker concept is that of bounded imaginary powers. 

Let $A:\scrD(A)=\X_1 \to \X_0$ be a closed unbounded linear operator satisfying \eqref{sectorial} for some angle $0<\theta<\pi$. 
%We can then define the purely  imaginary powers $A^{is}$, $s\in \R$, of $A$. 

\begin{definition}
$A$ is said to have bounded imaginary powers of angle $\phi$, if the powers $A^{is}$, $s\in \R$, are defined and 
$$\|A^{is}\|_{\scrL(\X_0)}\le e^{\phi |s|}, \quad s\in \R.$$
\end{definition} 

A bounded $H_\infty$-calculus for $A$ of angle $\theta<\pi/2$ implies the  boundedness of the imaginary powers for the same angle, since $z\mapsto z^{is}$  is bounded and holomorphic in 
%in $H_\infty(\C\setminus \Lambda_\theta)$. %
$\C\setminus \Lambda_\theta$, see \cite[Sect. I.2]{DHP}.
%In 1971, Seeley proved such an estimate for $L^p$ realizations of 

Still weaker is the concept of $\cR$-sectoriality. 

\begin{definition}\label{Rbdd}
A closed linear  operator  $A: \scrD(A)=\X_1\to \X_0$ satisfying \eqref{sectorial} is called  $\cR$-sectorial of angle $\theta$, if for all $N\in \N$,  $\lambda_{1},\ldots,\lambda_{N}\in \ \Lambda_{\theta}$ and  $x_{1},\ldots ,x_{N}\in \X_0$, 
\begin{eqnarray}\label{RS1}
\Big\|\sum_{\rho=1}^{N}\epsilon_{\rho}\lambda_{\rho}(\lambda_{\rho}-A)^{-1}x_{\rho}\Big\|_{L^{2}(0,1;X_0)} \leq C \;\Big\|\sum_{\rho=1}^{N}\epsilon_{\rho}x_{\rho}\Big\|_{L^{2}(0,1;X_0)},
\end{eqnarray}
for some constant $C\geq 1$, called the {\em $\cR$-bound}, and the sequence $\{\epsilon_{\rho}\}_{\rho=1}^{\infty}$ of the Rademacher functions, see \cite[Section 1.9]{KW04}.
\end{definition} 

That the existence of a bounded $H_\infty$-calculus implies $\cR$-sectoriality for the same angle  is shown in \cite[Theorem I.4.5]{DHP}. Moreover, we have the following theorem by Weis, see \cite[Theorem 4.2]{W}:

\begin{theorem}\label{T4.2} 
In a UMD Banach space any operator that is $\cR$-sectorial  -- in particular any operator with a bounded $H_\infty$-calculus or bounded imaginary powers -- of angle less than $\pi/2$ has maximal $L^q$-regularity for any $1<q<\infty$.  
\end{theorem} 

\begin{remark} \label{R4.3}
In order to show that $A$ has maximal $L^q$-regularity, it suffices to show this for $A+c$, $c>0$. 
In fact, given $f\in L^q(0,T;\X_0) $ and $u_0 \in (\X_0,\X_1)_{1-1/q,q}$ we then find a unique  $u_c\in W^1_q(0,T; \X_0) \cap L^q(0,T; \X_1)$ solving 
$$\partial_t u_c + (A+c)u_c= e^{-ct} f, \quad u_c(0) = u_0.$$
But then $u=e^{ct} u_c\in W^1_q(0,T; \X_0) \cap L^q(0,T; \X_1)$ is the unique solution to 
$$\partial_t u + Au= f, \quad u(0) = u_0.$$
\end{remark} 

%%%%%%%%%%%%%%%%%%%%%%%%
\subsection{The Laplace-Beltrami Operator}
%%%%%%%%%%%%%%%%%%%%%%%%

For some closed extensions $\underline \Delta$,  the operator  $-\underline\Delta+c$, $c>0$ large, has a bounded  $H_\infty$-calculus. In order to show this, we introduce the following notation for the asymptotics spaces $\widehat{\scrE}_{q_j^\pm}$ of $\widehat\Delta$ defined in \eqref{3.3} and \eqref{3.4}: 
Given a subspace $\widehat{\underline{\scrE}}_{q_j^\pm}$ of $\widehat\scrE_{q_j^\pm}$ of the form 
\begin{eqnarray*}
\widehat{\underline{\scrE}}_{q_j^\pm} = 
\{u\in C^\infty(\R_+\times \partial \B): u(x,y) = \omega x^{-q_j^\pm} e_j(y) ; e_j\in \underline E_j\subseteq E_j\} ,
\end{eqnarray*}
%for some fixed cut-off function $\omega$, 
where $\underline E_j$ is a subspace of the space  $E_j$ introduced in Example \ref{2.18}, we let $\underline E_j^\perp$ be the orthogonal complement of $\underline E_j$ in $E_j$ and define
for  $n\not=1$ or $j\not=0$
\begin{eqnarray*}
\widehat{\underline\scrE}_{q_j^\pm}^\perp = 
\{u\in C^\infty(\R_+\times \partial \B): u(x,y) = \omega x^{-q_j^\pm} e_j(y) : e_j\in \underline E_j^\perp\subseteq E_j\} .
\end{eqnarray*}
 For $n=1$ and $q_0^\pm=0$, when $\widehat\scrE_0$ is given by 
\eqref{3.4}, let  
$$
\widehat{\underline\scrE}_0^\perp =
\begin{cases}
\widehat{\scrE}_0, \quad \widehat{\underline\scrE}_0= \{0\},\\ 
\widehat{\underline\scrE}_0, \quad\widehat{\underline\scrE}_0 = \{\omega(x) e_0: e_0\in E_0\},\\
 \{0\},\quad \text{else}.
\end{cases}
$$
A key result then is the following theorem that goes back to \cite{SS1} and \cite{SS2}; see also \cite{GKM1} for resolvent estimates. 
\begin{theorem}
\label{HDelta}
Let $s\ge0$, $|\gamma|<(n+1)/2$  and
\begin{eqnarray}
\label{4.3}
%\lefteqn{
\frac{n+1}2-\gamma, \frac{n+1}2-\gamma-2\notin \{q_j^\pm: j=0,1,\ldots\}. 
\end{eqnarray}
Moreover let $\underline \Delta$ be a closed extension of $\Delta$ in $\cH^{s,\gamma}_p(\B)$ with domain 
$$\scrD(\underline\Delta) = \cH^{s+2,\gamma+2}_p(\B) \oplus \bigoplus_{q\in I_\gamma} \underline \scrE_q .
$$
with $I_\gamma$ defined in \eqref{3.6.1}. Denote by $\widehat {\underline\scrE}_q$  the associated spaces in the domain of $\widehat \Delta$ under the isomorphism $\Theta$ explained after Theorem \ref{T3.4}.  
Assume that, for $q\in I_\gamma\cap \{q_j^\pm: j\ge 0\}$, 
\begin{enumerate}\renewcommand{\labelenumi}{\rm(\roman{enumi})}
\item $\widehat{\underline{\scrE}}_q^\bot= \widehat{\underline\scrE}_{n-1-q} $ for $q\in I_\gamma\cap I_{-\gamma}$
\item $\widehat{\underline{\scrE}}_q  = \widehat{\scrE}_q$,  for $\gamma\ge0$ and $q\in I_\gamma\setminus I_{-\gamma}$
\item $\widehat{\underline{\scrE}}_q =\{0\}$ for $\gamma \le0$ and $q\in I_\gamma\setminus I_{-\gamma}$.
\end{enumerate} 
Then there exists a $c>0$ such that $-\underline\Delta+c$ has a bounded $H_\infty$-calculus  in $\cH^{s,\gamma}_p(\B)$ of any angle $\theta >0$.  
\end{theorem} 

\begin{proof}
This follows by combining Theorems 6.5 and 5.2 in \cite{SS2}.
\end{proof} 

\begin{example} \label{3.9} Suppose $\Delta$ is elliptic with respect to the line $\Gamma_{(n+1)/2-\gamma-2}$. We know that the maximal domains of $\Delta$ and $\widehat\Delta$ are of the form  \begin{eqnarray*}
\scrD( \Delta_{\max}) = \cH^{s+2,\gamma+2}_p(\B) \oplus {\scrE} \quad\text{and} \quad
\scrD( {\widehat\Delta}_{\max}) = \cK^{s+2,\gamma+2}_p(\R_+\times \partial \B) \oplus {\widehat\scrE}  
\end{eqnarray*}
for finite rank spaces $\scrE$ and $\widehat \scrE$. 
According to Corollary \ref{3.5} we obtain closed extensions $\underline \Delta$ and $\underline {\widehat\Delta}$ of 
$\Delta$ and $\widehat \Delta$, respectively, by letting, for any subspaces $\underline \scrE$ of $\scrE$ and  $\underline{\widehat\scrE}$ of $\widehat \scrE$,
\begin{eqnarray}\label{4.4}
\scrD(\underline \Delta) = \cH^{s+2,\gamma+2}_p(\B) \oplus \underline{\scrE} \text{ and } 
\label{4.5}
\scrD(\underline {\widehat\Delta}) \!&\!\!=& \cK^{s+2,\gamma+2}_p(\R_+\times \partial \B) \oplus \underline{\widehat\scrE}. 
\end{eqnarray}

In Theorem \ref{3.7}, with $\gamma$ fixed according to \eqref{gamma1},  we may choose 
\begin{eqnarray}\label{4.6}
\underline\scrE = \bigoplus_{j=1}^k \scrE_{q^-_j} \oplus \underline{\scrE}_0, \text{with }\;
\underline \scrE_0 = \{u:u(x,y) = \omega(x) e(y); e\in E_0\}
\end{eqnarray}
and $E_0$ defined in Example \ref{2.18}. The image of $\underline{\scrE}$ under the isomorphism $\Theta$ in \eqref{theta} is 
\begin{eqnarray} 
\label{4.7} 
\underline{\widehat\scrE} = \bigoplus_{j=1}^k \widehat\scrE_{q^-_j} \oplus \underline{\widehat{\scrE}}_0
\end{eqnarray}
with $\underline{\widehat\scrE}_0= \{u\in C^\infty(\R_+\times \partial \B): u(x,y) = \omega(x)e(y); e\in E_0\}$.
We then obtain:
\end{example} 
\begin{theorem}\label{T4.6}
Let $s\in \R$, $1<p<\infty$ and assume $\gamma$ satisfies \eqref{gamma1}.
% Let $s\in \R$, $1<p<\infty$  and assume $\gamma$ satisfies \eqref{gamma1} and \eqref{gamma2}. 
With  $\scrE$ as  in  \eqref{4.6}, 
the extension $\underline\Delta$ satisfies the assumptions of Theorem \ref{HDelta}, and hence there exists a $c>0$ such that $-\Delta+c$ has a bounded $H_\infty$-calculus on $\cH^{s,\gamma}_p(\B)$ for any angle $\theta>0$.
\end{theorem} 

\begin{proof} First let $s\ge0$.
Using that $n\ge 1$ and $-2<(n+1)/2-\gamma-2<0$, it is easily seen that $|\gamma|< (n+1)/2$.
Moreover, the choice of $\gamma$ implies that \eqref{4.3} holds. 
It remains to check conditions (i), (ii) and (iii) in Theorem \ref{HDelta}. 

To this end we first note that, for $q_0^-\not=q_0^+$, the domain of $\widehat{\underline {\Delta}}$ comprises  the full asymptotics spaces $\widehat{\scrE}_{q_j^-}$ over the $q_j^-$ and the zero spaces over the $q_j^+$, while, in case $q_0^+=q_0^-=0$, the choice implies that $\widehat{\underline \scrE}_0 = \widehat{\underline \scrE}_0^\perp$. Since $q_j^-= n-1-q_j^+$,  $I_\gamma\cap I_{-\gamma}$ contains either both $q_j^+$ and $q_j^-$ or neither of them. Hence (i) is fulfilled. 
Since, for $\gamma>0$, the interval $I_\gamma$ lies to the left of $I_{-\gamma}$ on the real axis, while, for $\gamma<0$, it lies to the right, a similar consideration  shows that (ii) and (iii) hold. 

For negative $s$ consider the adjoints $\underline{\Delta}^*$ and $\underline{\widehat\Delta}^*$. The domain of $\widehat{\underline\Delta}^*$ has  been determined in \cite[Theorem 5.3]{SS1} and one can check the conditions (i), (ii) and (iii) as before. 
\end{proof}

\begin{example}\label{3.10} 
In Theorem \ref{T4.6} we have chosen the base space $\cH^{s,\gamma}_p(\B)$ in such a way that the  maximal extensions of $\Delta$ and $\widehat \Delta $ contain as many asymptotics spaces $\scrE_{q_j^-}$ as possible. On the one hand, this gives us very precise information on the possible asymptotics a solution can develop. On the other hand, it considerably restricts the interpolation space $(\cH^{s,\gamma}_p(\B), \scrD(\underline \Delta))_{1-1/q,q}$ of admissible initial values in the Cl\'ement-Li Theorem; see Section \ref{sect.interpolation}, below, for more details on this space.   
Depending on the situation, the following choices may be of interest: 

Fix $\gamma$ such that $(n+1)/2-\gamma-2, (n+1)/2-\gamma\not=q_j^\pm$  for all $j$ and
\begin{eqnarray}\label{4.8}
\max \{-2, q^-_{\ell+1}\}  <\frac{n+1}2-\gamma-2< q_\ell^- \le 0 
\end{eqnarray}
for some $\ell \ge0$.
%Depending on the choice of $\gamma$ and the location of the $q_j^+$,  the maximal domains may contain unbounded functions  from the $q_j^+$. However, we can 
Consider the extension $\underline \Delta$
% and $\widehat{\underline{\Delta}}$ of $\Delta$ and $\widehat \Delta$, respectively, 
with domain 
\begin{eqnarray}\label{4.9} 
\scrD(\underline\Delta) = \cH^{s+2,\gamma+2}_p(\B) \oplus \bigoplus_{j=1}^\ell \scrE_{q_j^-}\oplus \underline{\scrE}_0
%\quad \text{and} \\
%\scrD(\widehat{\underline\Delta}) = \cK^{s+2,\gamma+2}_p(\B) \oplus \bigoplus_{j=1}^\ell \scrE_{q_j^-}\oplus \underline{\scrE}_0.
%$=\underline {\widehat\scrE}_0
\end{eqnarray}
with $\underline\scrE_0$ as in \eqref{4.6}. %As before, $\underline \scrE_0  = \{u(x,y) = \omega(x) e(y): e\in E_0\}$. 
Here, $\ell$ can be chosen to be zero; then apart from the constants no further asymptotics terms arise. 
In particular, for $q_1^-\le -2$,  this is necessarily the case. 

Under the isomorphism $\Theta$, the spaces  $\scrE_{q_j^-}$, $j\ge1$, are mapped to  
${{\widehat\scrE}}_{q_j^-}$ and  $\underline \scrE_0$ to $\underline {\widehat\scrE}_0$. 
The same considerations as in the proof of Theorem \ref{T4.6} then show the following theorem. \end{example}   

\begin{theorem} \label{T4.8}
The extension $\underline\Delta$ with domain \eqref{4.9} also satisfies the assumptions of Theorem \ref{HDelta}. Hence there exists a $c>0$ such that $-\underline\Delta+c$ has a bounded $H_\infty$-calculus on $\cH^{s,\gamma}_p(\B)$ for any $s\in \R$ and  any angle $\theta>0$. In particular, $-\Delta$ has maximal regularity by  Theorem \ref{T4.2} and Remark \ref{R4.3}. 
\end{theorem}

%%%%%%%%%%%%%%%%%%%%%%%%
\subsection{Interpolation} \label{sect.interpolation}
%%%%%%%%%%%%%%%%%%%%%%%%
Consider the extensions $\underline \Delta$ of $\Delta$ in $\cH^{s,\gamma}_p(\B)$ in Theorems \ref{T4.6} and \ref{T4.8}.  In the Cl\'ement-Li Theorem, the real interpolation spaces $(\X_0,\X_1)_{1-1/q,q}$, $1<q<\infty$,  play an important role. In the case at hand we have to consider
$$(\cH^{s,\gamma}_p(\B), \scrD(\underline\Delta))_{1-1/q,q}=\Big(\cH^{s,\gamma}_p(\B), \cH^{s+2,\gamma+2}_p(\B)\oplus \bigoplus_{j=1}^\ell \scrE_{q_j^-} \oplus \underline \scrE_0\Big)_{1-1/q,q}.$$
For these, we have the following result. 
%Moreover, we have the following extension of \eqref{embedding1}:
\begin{lemma}\label{interpolation}Let $s\in \R$, $1<p,q<\infty$ and $\gamma$ be as in \eqref{4.8}. 
For every $\varepsilon>0$ we have continuous and dense embeddings
\begin{eqnarray}\label{embedding2}
\cH^{s+2-2/q+\varepsilon, \gamma+2-2/q+\varepsilon}_p(\B) &+&\bigoplus_{j=1}^\ell\scrE_{q_j^-}+\underline{\scrE}_0 \hookrightarrow 
(\cH^{s,\gamma}_p(\B), \scrD(\underline\Delta))_{1-1/q,q}\nonumber\\
&\hookrightarrow&
\cH^{s+2-2/q-\varepsilon, \gamma+2-2/q-\varepsilon}_p(\B)+ \bigoplus_{j=1}^\ell \scrE_{q_j^-}\oplus \underline\scrE_0.
\end{eqnarray}

{\rm(i)} The sum on the right hand side is direct when $\frac{n+1}2-\gamma-2+\frac2q +\varepsilon \le q_\ell^-$, which can be achieved in view of \eqref{4.8} by taking $q$ large and $\varepsilon$ small. 

{\rm(ii)} For general $q$ we find an index $0\le r\le \ell$ such that 
$$\max\{-2,q_{r+1}^-\} <\frac{n+1}2-\gamma-2+\frac2q+\varepsilon< q_r^-$$ for all sufficiently small $\varepsilon>0$. Then the right hand side is 
$$\cH^{s+2-2/q-\varepsilon, \gamma+2-2/q-\varepsilon}_p(\B)\oplus \bigoplus_{j=1}^r \scrE_{q_j^-}\oplus\underline\scrE_0.$$
\end{lemma} 

A complete {\em proof} can be found in \cite{RS5}. Here is a sketch: The left embedding is a consequence of Theorem \ref{2.10} together with the embedding results for complex and real interpolation and the fact that the asymptotics spaces embed into $\scrD(\underline\Delta) $. For the right embedding we also use Theorem \ref{2.10}. Moreover, we observe that $\Delta(\scrD(\underline\Delta))\subseteq \cH^{s,\gamma}_p(\B))$. Together with Theorem \ref{2.10} this implies that, for every $\varepsilon>0$, 
\begin{eqnarray*}
\Delta((\cH^{s,\gamma}_p(\B), \scrD(\underline\Delta))_{1-1/q,q}) &\hookrightarrow &
(\cH^{s-2,\gamma-2}_p(\B), \cH^{s,\gamma}_p(\B)))_{1-1/q,q}\\
&\hookrightarrow& \cH^{s-2/q-\varepsilon,\gamma-2/q-\varepsilon}_p(\B).
\end{eqnarray*}
In particular, the interpolation space embeds into the maximal domain of $\Delta$ as an unbounded operator in $\cH^{s-2/q-\varepsilon,\gamma-2/q-\varepsilon}_p(\B)$, 
which gives the right hand side. 

Statement (i) is an immediate consequence of the fact that a non-zero function in $\scrE_{q_j^-}$
%which, close to the boundary is of the form  $x^{-q}\ln^kxc(y)$ with $q\in \R$, $k\in \N_0$ and $c\in C^\infty(\partial \B)$ 
belongs to $\cH^{s',\gamma'}_{p}(\B)$ if and only if $q_j^-<(n+1)/2-\gamma'$. 
(ii) is clear. \hfill $\Box$

%%%%%%%%%%%%%%%%%%%%%%%%
\section{Example: The Porous Medium Equation} 
%%%%%%%%%%%%%%%%%%%%%%%%
The porous medium equation (PME) models the flow of a gas in a porous medium. Denoting by $u$ the density of the gas, it takes the form 
\begin{eqnarray}
\label{PME}
\partial_t u -\Delta (u^m) &=& 0\\% F(t,u) \\
\label{IV}
u|_{t=0}&=& u_0.
\end{eqnarray}
Here $t$ denotes time and  $m$ is a positive constant.  It is a generalization of the heat equation which is obtained for $m=1$. Instead of setting the right hand side of \eqref{PME} to zero, one may include a nonlinear forcing term $F=F(t,u)$. Vazquez' monograph  \cite{V07} contains a wealth of information; see \cite{AP81} by Aronson and Peletier for an instructive  introduction. 

%%%%%%%%%%%%%%%%%%%%%%%%
\subsection{Short Time Solutions}
%%%%%%%%%%%%%%%%%%%%%%%%

We consider the PME on $\B$ with the Laplace-Beltrami operator for the conically degenerate metric $g$. 
The existence of short time solutions has been established in  \cite[Theorem 6.5]{RS1}:
%We consider the Laplacian as an unbounded linear operator in $\cH^{s,\gamma}_p(\B)$ with $s,\gamma $ and $p$ determined, below. In view of the fact that the domain of $\Delta$ should include functions that are nonzero and bounded near the boundary, we choose the extension with domain 
%\begin{eqnarray*}
%\scrD(\underline \Delta) =  \cH^{s+2,\gamma+2}_p(\B) \oplus \underline\scrE_0.
%\end{eqnarray*}
%Choosing $\gamma$ as in \eqref{4.8} and
%$$ \underline \scrE_0  = \{u(x,y) = \omega(x) e(y): e\in E_0\}$$
%this  is a domain as considered in Example \ref{3.10} for $\ell=0$. 
%
%Theorem  6.5 in \cite{RS1} then yields: 

\begin{theorem}\label{T5.1}
Choose $\gamma$ such that 
\begin{eqnarray}\label{eq.5.1}
\max\{-2, q_1^-\}<(n+1)/2-\gamma-2 <0 
\end{eqnarray}
and fix $1<p,q<\infty$ so large that 
\begin{eqnarray}\label{5.3}
%\label{}
%\lefteqn{
\frac{n+1}p +\frac2q<1\quad\text{and}\quad \frac{n+1}2-\gamma-2+\frac4q<0.
\end{eqnarray}

For $s>-1+(n+1)/p + 2/q$ consider the extension $\underline \Delta $ of $\Delta$  with the domain 
\begin{eqnarray}\label{Dom}
\scrD(\underline \Delta) =  \cH^{s+2,\gamma+2}_p(\B) \oplus \underline\scrE_0,
\end{eqnarray}
where  
$ \underline \scrE_0$ is as in \eqref{4.6}. %  = \{u(x,y) = \omega(x) e(y): e\in E_0\}.$$
Given any strictly positive 
$$u_0\in (\cH^{s,\gamma}_p(\B), \scrD(\underline\Delta))_{1-1/q,q},$$
the initial value  problem for the porous medium equation in
$L^q(0,T_0;\cH^{s,\gamma}_p(\B))$ 
\begin{eqnarray}
%\label{}
%\lefteqn{
\partial_t u -\Delta(u^m) = F(t,u), \quad u|_{t=0}=u_0,
\end{eqnarray}
with forcing term $F=F(t,\lambda)\in C^{1-,1-}([0,T_0]\times U_0;  \cH^{s,\gamma}_p(\B))$ for a neighborhood $U$ of $u_0$ in $(\cH^{s,\gamma}_p(\B), \scrD(\underline\Delta))_{1-1/q,q}$,
has a unique solution 
\begin{eqnarray*}
u\in W^1_q(0,T; \cH^{s,\gamma}_p(\B)) \cap L^q(0,T; \scrD(\underline \Delta))
\end{eqnarray*}
for suitable  $T\le T_0$. 
\end{theorem} 

The choice of $\gamma$ implies that the second condition in \eqref{5.3} can always be fulfilled. 
In \cite{RS2} a stronger condition on $F$ has been used. It turns out that the present one suffices. \smallskip 

%There is one feature of the above domain that was observed in \cite[Lemma 6.2]{RS2} and that we will need in the proof:
%
%\begin{lemma}\label{si}  It follows from \ref{multiplier} that $\cA= \cH^{s_0,\gamma_0}_p(\B)\oplus \underline\scrE_0$ for $s_0>(n+1)/p$ and $\gamma_0>(n+1)/2$ is a unital Banach algebra. 
%Moreover, it is spectrally invariant, i.e., if  $u\in \cH^{s_0,\gamma_0}_p(\B)\oplus \underline\scrE_0$, be pointwise invertible. Then $u^{-1} \in \cH^{s_0,\gamma_0}_p(\B)\oplus \underline\scrE_0$.
%As a consequence, it is closed under holomorphic functional calculus. 
%\end{lemma}

%%%%%%%%%%%%%%%%%%%%%%%%
{\bf Idea of the proof of Theorem \ref{T5.1}}.
%%%%%%%%%%%%%%%%%%%%%%%%
We write 
%want to apply the Cl\'ement-Li Theorem. To this end we write 
\begin{eqnarray}\label{Deltaum}
\Delta(u^m) = mu^{m-1} \Delta u + m(m-1)u^{m-2} \langle\nabla u, \nabla u\rangle_g,
\end{eqnarray} 
\begin{eqnarray*}
\nabla u = \sum_{j,k} g^{jk}\frac{\partial u}{\partial z_j }\frac\partial{\partial z_k}, \quad z\in \Int\B,
\end{eqnarray*}
and note that, in coordinates $(x,y)$ near $\partial \B$,
\begin{eqnarray}\label{nabla}
\langle\nabla u, \nabla v\rangle_g= \frac1{x^2}\Big((x\partial_x u)(x\partial_xv) + \sum_{j,k=1}^nh^{jk}
\frac{\partial u}{\partial y_j} \frac{\partial v}{\partial y_k}\Big).
\end{eqnarray}
Here $(g^{jk})$ and $(h^{jk})$ are the inverses  of the metric tensors $g$ and $h$ in \eqref{eq.g}. 

Then the porous medium equation takes the form %This allows us to write the PME in the form 
\begin{eqnarray*}
\partial_t u +A(u)u =f(t,u)
\end{eqnarray*}
with  $A(u) = -mu^{m-1} \underline \Delta$ and 
$f(t,u) = F(t,u)+m(m-1) u^{m-2} \langle\nabla u, \nabla u\rangle_g$.
This suggests to apply  Theorem \ref{ThmCL} with 
\begin{eqnarray*}
\X_0 = \cH^{s,\gamma}_p(\B) \quad \text{and}\quad \X_1 = \scrD(\underline\Delta) = 
\cH^{s+2,\gamma+2}_p(\B) \oplus \underline \scrE_0. 
\end{eqnarray*}
We choose the open neighborhood  $U$ of $u_0$ in the Cl\'ement-Li theorem as follows: 
\begin{eqnarray*}
U= \{u\in  (\cH^{s,\gamma}_p(\B), \scrD(\underline\Delta))_{1-1/q,q}: \frac12\min_\B u_0 <u<2\max_\B u_0. 
\end{eqnarray*}

In order to establish the existence of a short time solution it  suffices to show that 
\begin{enumerate}\renewcommand{\labelenumi}{(\roman{enumi})}
\item $A(u_0) $ has maximal regularity,  
\item Condition (H1) and (H2) in Theorem \ref{ThmCL} are fulfilled. 
\end{enumerate} 

(i) With our choice of $\gamma$ and $\scrD(\underline\Delta)$ we are in the situation of Example \ref{3.10} with $\ell=0$.  
Theorem \ref{T4.8} therefore asserts that $c-\underline\Delta$ has a bounded $H_\infty$-calculus for suitable $c>0$.
 
By Lemma \ref{interpolation} 
the interpolation space $(\cH^{s,\gamma}_p(\B), \scrD(\underline \Delta))_{1-1/q,q}$ 
embeds into $\cA :=\cH^{s+2-2/q-\varepsilon, \gamma+2-2/q-\varepsilon}_p(\B)\oplus \underline\scrE_0$ for arbitrary $\varepsilon>0$. Our assumptions on $s$, $p$ and $q$ imply that $s+2-2/q>|s| +(n+1)/p$ and $\gamma+2-2/q>(n+1)/2$. 

According to Lemma \ref{si}, below, 
$\cA$ is spectrally invariant and therefore closed under holomorphic functional calculus.  
Hence, if the strictly positive function $u$ belongs to $\cA$; 
the same is true for the function $mu^{m-1}$.
% also is an element of $\cH^{s+2-2/q-\varepsilon, \gamma+2-2/q-\varepsilon}_p(\B)\oplus \underline\scrE_0$. 
Moreover,  Lemma \ref{multiplier} then  implies that $mu^{m-1}$ is a bounded multiplier on $\cH^{s,\gamma}_p(\B)$. 
Hence 
 $$mu^{m-1}\Delta: \scrD(\Delta) \to \cH^{s,\gamma}_p(\B)$$ 
is a bounded linear operator. Since $c-\underline \Delta$ has a bounded $H_\infty$-calculus of any angle $\theta>0$, it also is $\cR$-sectorial for any angle $\theta$. A lengthy perturbation argument, see  \cite[Theorem 6.1]{RS2}, based on a freezing of coefficients, then shows that the operator  $c-mu^{m-1}\underline\Delta$ also is $\cR$-sectorial of any angle $\theta>0$. In particular,  $mu^{m-1}\underline\Delta$ has maximal regularity.  

(ii) Checking  (H1) and (H2) requires more delicate multiplier arguments, see \cite{RS2}.\hfill $\Box$

\begin{lemma}\label{si}  It follows from Lemma \ref{multiplier} that, for $s_0>(n+1)/p$ and $\gamma_0\ge (n+1)/2$,  $\cA= \cH^{s_0,\gamma_0}_p(\B)\oplus \underline\scrE_0$ is a unital Banach subalgebra of $C(\B)$ with a stronger topology.
Moreover, it is spectrally invariant: If  $u\in \cA$ is  invertible in $C(\B)$, then $u^{-1} \in \cA$.
In particular,   $\cA$  is closed under holomorphic functional calculus.
%: If $f$ is  holomorphic in a neighborhood of the spectrum of $u$ and $\gamma$ is a contour that simply surrounds the spectrum, then  
\end{lemma}

\begin{proof}This is \cite[Lemma 2.5(c)]{RS2}. The closedness under holomorphic functional calculus can be seen as follows. Let $f$ be holomorphic near the spectrum of $a\in \cA$. For a contour $\scrC$ simply surrounding the spectrum of $a$ we can define 
$$f(a) =(2\pi i)^{-1}  \int_{\scrC} f(\lambda) (\lambda-a)^{-1} \, d\lambda.$$
Spectral invariance shows that $(\lambda-a)^{-1} \in \cA$. It also implies the continuity of inversion by a result of Waelbroeck  \cite{Wa3}, so that the integral converges in $\cA$.   Hence $f(a)\in \cA$. 
\end{proof}

%%%%%%%%%%%%%%%%%%%%%%%%
\subsection{Global Solutions}  
%%%%%%%%%%%%%%%%%%%%%%%%
A natural question is whether the solution found in Theorem \ref{T5.1} exists for all times. 
It was answered affirmatively in \cite[Theorem 1.1]{RS3}, namely: 

\begin{theorem}\label{T5.2}
Under the conditions of Theorem \ref{T5.1}, the initial value for the PME without forcing term
\begin{eqnarray}\label{IPME}
%\label{}
%\lefteqn{
\partial_t u -\Delta(u^m) =0, \quad u|_{t=0}=u_0,
\end{eqnarray}
has, for all $T>0$, a unique solution 
$$u\in W^1_q(0,T:\cH^{s,\gamma}_p(\B)) \cap L^q(0,T; \scrD(\underline\Delta)).$$
\end{theorem}

The proof of  Theorem \ref{T5.2} relies on two additional results, namely a comparison principle 
and a result on uniform H\"older continuity. %For completeness, we state them below. 

\begin{theorem}[Comparison principle]\label{T5.3}
Let $s$, $\gamma$, $p$, $q$ be  as in Theorem \ref{T5.1}, and let $u$ and $v$ in  $W^1_q(0,T;\cH^{s,\gamma}_p(\B)) \cap L^q(0,T; \scrD(\underline\Delta))$ 
be solutions to \eqref{IPME} for initial values $u_0, v_0\in (\cH^{s,\gamma}_p(\B), \scrD(\underline\Delta))_{1-1/q,q}$ and some $T>0$.   
Then $u\le v$ on $[0,T]\times\mathbb{B}$, if $u_0\le v_0$.

In particular, if  $0<M_{\min}\le u_0\le M_{\max}$ for suitable constants $M_{\min}$ and $M_{\max}$, then 
$M_{\min}\le u\le M_{\max}$ on $[0,T]\times\mathbb{B}$.
\end{theorem}

This is \cite[Theorem 4.2]{RS3}. The {\em proof} was adapted from the proof in the classical case, see Theorem 6.5 in \cite{V07}.

\begin{theorem}[H\"older continuity] \label{T5.4}
Let $s$, $\gamma$, $p$, $q$ be  as in Theorem \ref{T5.1}, and  let $u$ in  
$W^1_q(0,T:\cH^{s,\gamma}_p(\B)) \cap L^q(0,T; \scrD(\underline\Delta))$  be a solution of \eqref{IPME} with initial value $u_0$. Then $u$ is H\"older continuous on ${[0,T[} \times\B$, i.e. $u\in C^\alpha({[0,T[}\times\B)$ for some $\alpha>0$. Both $\alpha$ and $\|u\|_{C^\alpha({[0,T[} \times\B)}$ depend only on the dimension $n$, the bounds $M_{\min}$ and $M_{\max}$ for the initial value $u_0$ and the metric $g$, not on $T$. 
\end{theorem}

This is an adaptation of a result by Lady\v zenskaya, Solonnikov and Ural'ceva in  Section II.8 of \cite{LSU} to the case of manifolds with conical singularities, see \cite[Theorem 4.4]{RS3}. The proof was inspired by Amann's proof of \cite[Lemma 5.1]{Am3}.  

We can now  complete the proof of Theorem \ref{T5.2}: The comparison principle implies that  $M_{\min}$ and $M_{\max}$ in Theorem \ref{T5.3} will bound the solution $u$ as long as it exists. Hence Theorem \ref{T5.4} implies uniform H\"older bounds for $u$ for all times.  This allows to extend the solution, initially defined on a small interval ${[0,T'[}$, iteratively to the given interval  ${[0,T[}$. \hfill $\Box$

%%%%%%%%%%%%%%%%%%%%%%%%
\subsection{Asymptotics near the Conical Points}
%%%%%%%%%%%%%%%%%%%%%%%%

Another interesting  question is how the shape of the cross-section of the cone influences the solutions to the porous medium equation, in other words, how the gas enters the tip of the cone. To find this out, one can start with an initial value $u_0$ which, near $\partial \B$, is of the form 
$$u_0 = c_0  + \cO(x^\infty)$$
for a constant $c_0>0$ and look how it evolves. This problem was addressed in \cite{RS5}. 

We briefly recall the setting. We use the extension $\underline\Delta$ of $\Delta$ in Example \ref{3.9} and  choose $k\in \N_0$ such that 
$q^-_{k+1} \le -2<q_k^- $ and $\gamma$ with 
\begin{eqnarray}\label{5.7}
-2<\frac{n+1}2-\gamma-2<q_k^-\quad \text{and} \quad \frac{n+1}2-\gamma\not= q_j^+
\end{eqnarray}
for all $j$. Then we fix, for arbitrary $s$ and $1<p<\infty$,  
\begin{eqnarray}\label{5.8} 
\scrD(\underline\Delta) = \cH^{s+2,\gamma+2}_p (\B) \oplus \bigoplus_{j=1}^k\scrE_{q_j^-} \oplus \underline\scrE_0,
\end{eqnarray}
where, as before,  $\underline\scrE_0 =\{u: u(x,y) =\omega(x) e(y) ; e\in E_0\}$. 

In case  $q_1^-\le -2$, the sum over the $\scrE_{q_j^-}$, $j=1,\ldots,k$, is void and  we obtain the same extension as in Theorem \ref{T5.1}. 
In general, however, this domain may be smaller due to the restriction on $\gamma$ in \eqref{5.7}. It contains the maximal number of asymptotics spaces. 
We therefore expect to see more asymptotics terms in the solution. This turns out to be true. 
However, we will have to use a different approach:
% as it is in general not possible to check condition (H2) in the Clément-Li Theorem for the map  
We recall that the substitution $v=u^m$  
transforms the initial value problem 
\begin{eqnarray*}
%\label{}
\partial_t u -\Delta(u^m) = F(t,u), \quad u|_{t=0}=u_0,
\end{eqnarray*}
into 
\begin{eqnarray}
\label{PMEAlt}
%\lefteqn{
\partial_tv -mv^{\frac{m-1}m}\Delta v = G(t,v) , \quad v|_{t=0} = v_0,
\end{eqnarray}
with $G(t,v) = m v^{(m-1)/m}F(t,v^{1/m})$ and $v_0 = u_0^m$.
The following is Theorem 1.4 in \cite{RS5}.

%Theorem 1.4 in \cite{RS5} is an analog of Theorem \ref{T5.1}. It states

\begin{theorem}\label{T5.7}
Choose $\gamma$ as in \eqref{5.7} and fix $1<p,q<\infty$ so large that 
\begin{eqnarray*}%\label{6.3}
%\label{}
%\lefteqn{
\frac{n+1}p +\frac2q<1\quad\text{and}\quad \frac{n+1}2-\gamma-2+\frac4q<0.
\end{eqnarray*}

For $s>-1+(n+1)/p + 2/q$ consider the extension $\underline \Delta $ of $\Delta$  with domain 
\eqref{5.8}.
Let
$v_0\in (\cH^{s,\gamma}_p(\B), \scrD(\underline\Delta))_{1-1/q,q}$
be strictly positive on $\B$ and let $U$ be an open neighborhood of $v_0$ in  $(\cH^{s,\gamma}_p(\B), \scrD(\underline\Delta))_{1-1/q,q}$. 

Then  the initial value  problem %for the porous medium equation 
in $L^q(0,T_0;\cH^{s,\gamma}_p(\B))$ 
\begin{eqnarray}
%\label{}
\partial_t v -mv^{\frac{m-1}m}\Delta(v) = G(t,v), \quad v|_{t=0}=v_0,
\end{eqnarray}
with  $G\in C^{1-,1-}([0,T_0]\times U; \cH^{s,\gamma}_p(\B) )$ has a unique solution $v$ for some  $T\le T_0$,  and
\begin{eqnarray}\label{5.10} 
v\in W^1_q(0,T; \cH^{s,\gamma}_p(\B)) \cap L^q(0,T; \scrD(\underline \Delta)).
\end{eqnarray}
\end{theorem} 

\begin{proof} Theorem \ref{T4.6} shows that $c-\underline\Delta$ has a bounded $H_\infty$-calculus in $\cH^{s,\gamma}_p(\B)$ for any angle $\theta>0$. From there, one can argue as in the proof of Theorem \ref{T5.1}. 
\end{proof}

An application of Lemma \ref{interpolation} shows that, with $r$ chosen as in Lemma \ref{interpolation}(ii),
\begin{eqnarray}\nonumber
 v&\in& C([0,T]; (\cH^{s,\gamma}_p(\B), \scrD(\underline\Delta))_{1-1/q,q})\\
 &\subseteq& 
C([0,T]; \cH^{s+2-2/q-\varepsilon,\gamma+2-2/q-\varepsilon}_p(\B)\oplus \bigoplus_{j=1}^r \scrE_{q_j^-}\oplus \underline\scrE_0).\label{Asymp_v}
\end{eqnarray}

To understand the  asymptotics of $u=v^{1/m}$  we introduce a new algebra: 

\begin{definition}
\label{scrB}
Let $1<p<\infty$, $ s_0>(n+1)/p$,  $ \gamma_0>(n+1)/2$ and $r\in \{0,\ldots,k\}$ with
$%\begin{eqnarray*}
\max\{-2,q_{r+1}^-\} <(n+1)/2- \gamma_0< q_r^-.
$%\end{eqnarray*}
Denote by $\scrB$ the subalgebra of $C(\B)$ generated by the functions in $\cH^{s_0, \gamma_0}_p(\B)$, $\scrE_{q_j^-}$, $1\le j\le r$, and $\underline\scrE_0$. 
%Recall that all these functions are continuous on $\B$. 
\end{definition}
 
\begin{theorem}\label{SIB}
$\scrB$ is spectrally invariant in $C(\B)$ and hence closed under holomorphic functional calculus.  
\end{theorem}

This is a refined version of Lemma \ref{si}. %, which we obtain if $r=0$. 
For a proof see \cite[Theorem 1.8]{RS5}.  
An element   $w\in \scrB$ can be written in the  form $w= u+p$, where $u\in \cH^{s_0,\gamma_0}_p(\B)$ and $p$ is a polynomial in functions in $\underline \scrE_0$, $\scrE_{q_1^-},\ldots, \scrE_{q_r^-}$. We can confine ourselves to products with at most $N$ factors in $\scrE_{q_1^-}, \ldots, \scrE_{q_r^-}$, where $N$ is the largest integer such that $(n+1)/2-\gamma_0<Nq_1^-$\; (recall that $q_1^-<0)$, as products with more of these factors will belong to $\cH^{s_0,\gamma_0}_p(\B) $. \smallskip

We can now answer the question raised above:

\begin{corollary} Consider the initial value problem  
$$\partial_t u -\Delta (u^m) = 0, \quad u(0) = u_0,$$
with  $u_0$ of the form 
$u_0 = c_0+w_0,$
where $c_0>0$ is a constant and $0\le w_0\in \cH^{s+2,\gamma+2}_p(\B)$. The substitution $v=u^m$ furnishes  a solution $v$ with  asymptotics \eqref{Asymp_v}. Hence $u=v^{1/m}$ 
will belong to the space  $C([0,T]; \scrB)$,  where we consider $\scrB$ for $s_0 = s+2-2/q-\varepsilon $, $\gamma_0=\gamma+2-2/q-\varepsilon $ with  small $\varepsilon>0$ and use  that  $v\mapsto v^{1/m}$ is continuous on the group  of invertible elements of $\scrB$. 
\end{corollary}

%%%%%%%%%%%%%%%%%%%%%%%%
\subsection{Further Results}
%%%%%%%%%%%%%%%%%%%%%%%%
{\bf Weak solutions}. So far, we have considered only the case of a strictly positive initial value $u_0$. In order to treat the porous medium equation also for the case, where  $u_0$
possibly vanishes on some part of $\B$, a concept of weak solutions, modeled on Aronson and Peletier's definition in \cite[Section 3]{AP81}, was introduced in \cite{RS3}:

\begin{definition}\label{defweak}
For $T>0$, let $\mathbb{B}_{T}=\mathbb{B}\times {]0,T]}$. A   weak solution of the initial value problem \eqref{IPME} is a function $v:\overline{\mathbb{B}}_{T}\rightarrow [0,\infty[$ such that
\begin{itemize}
\item[(i)] $\nabla v^{m}$ exists in the sense of distributions in $\mathbb{B}_{T}$ and 
$$
\int_{\mathbb{B}_{T}}(|v|^{2}+\langle\nabla v^m,\nabla v^m\rangle_g)dtd\mu_{g}<\infty,
$$
where  $d\mu_{g}$ denotes the  measure induced by the metric $g$.
\item[(ii)]\ For any $\phi\in C^{1}(\mathbb{B}_{T})$ such that $\phi=0$ on 
$\mathbb{B}\times \{T\}$, 
$$
\int_{\mathbb{B}_{T}}(\langle\nabla\phi,\nabla v^{m}\rangle_{g}-(\partial_{t}\phi)v)dtd\mu_{g}=\int_{\mathbb{B}}\phi(z,0)v(z,0)d\mu_{g}.
$$
\end{itemize}
\end{definition}

This led to the following result: 

\begin{theorem}[Existence of a weak solution]
Let $m\ge 1$, $s=0$ and $\gamma$, $p$, $q$ as in Theorem \ref{T5.1}. Furthermore, let
\[
u_{0}\in(\mathcal{H}_{p}^{0,\gamma}(\mathbb{B}), \mathcal{H}_{p}^{2,\gamma+2}(\mathbb{B})\oplus\underline{\mathscr{E}}_0,)_{1-1/q,q}
\]
be non-negative on $\mathbb{B}$. Then, for any $T>0$, the initial value problem \eqref{IPME} for the porous medium equation possesses a unique weak solution $u$ on $\mathbb{B}_T$.
\end{theorem}

%%%%%%%%%%%%%%%%%%%%%%%%
 
\noindent{\bf Boundary value problems}. Based on  work by Coriasco, Schrohe and Seiler in \cite{CSS1} and \cite{CSS2}, the porous medium equation was studied in \cite{RSS} on manifolds with boundary and conical singularities.  Theorem 7.7 in \cite{RSS} gives an analog of Theorem \ref{T5.1} for the initial value problem with Neumann boundary conditions. 
The analysis of boundary problems on manifolds with conical singularities, however, is much more involved, see e.g.,  \cite{SchulzeWiley}. \medskip

%%%%%%%%%%%%%%%%%%%%%%%%

\noindent{\bf The Yamabe flow on conic manifolds.}
%The study of the Yamabe flow is a classical problem of differential geometry. One is looking for 
The problem here is to find a family of metrics $g(t)$, $t\ge0$, satisfying 
\begin{eqnarray}\label{eq.Y}
\partial_t g(t) + R_{g(t)}g(t) = 0, \quad g(0) = g_0,
\end{eqnarray}
where $R_{g(t)}$ is the scalar curvature of $g(t)$. In  \cite{RY} Roidos proved the existence of short time solutions in the conformal class of the initial metric $g_0$ which he took to be the conically degenerate metric $g$ in \eqref{eq.g} for $n\ge2$. He assumed, moreover, that the scalar curvature of the metric $h(0)$ on $\partial\B$ in \eqref{eq.g} is $n(n-1)$.  

Writing $ g(t) = u^{4/(n-1)}(t) g_0$, \eqref{eq.Y} becomes an initial value problem for  $u$:   
\begin{eqnarray}
\label{eq.CYam}
\partial_tu -nu^{-\frac4{n-1}}\Delta u = -\frac{n-1}4 u^{\frac{n-5}{n-1}}R_{g_0}, \quad u(0) =1.  
\end{eqnarray}
This problem is very similar to that derived for the porous medium equation  in the proof of Theorem \ref{T5.1}, and it can be treated by the same method. Roidos finds: 
\begin{theorem}
Let $n\ge2$, $s\ge0$ and fix  $\gamma$ with  $\min\{-1,q_1^-\} <(n+1)/2-\gamma-2 <0$. Then choose $1<p,q<\infty$ such that 
%$$(n+1)/p+ 2/q<1 \quad\text{and} \quad (n+1)/2-\gamma-2+2/q<0.$$
 $$\frac{n+1}p+ \frac2q<1 \quad\text{and} \quad \frac{n+1}2-\gamma-2+\frac2q<0.$$
 Consider the extension $\underline\Delta$  of $\Delta$ in  $\cH^{s,\gamma}_p(\B)$ with domain  \eqref{Dom}. 
Then there exists a unique solution $u$ to \eqref{eq.CYam} in 
$$W^1_q(0,T; \cH^{s,\gamma}_p(\B)) \cap L^q(0,T; \cH^{s+2,\gamma+2}_p(\B) \oplus \underline\scrE_0).$$ 
\end{theorem}
This then furnishes a solution to the Yamabe problem \eqref{eq.Y}. 

There is a subtlety in the proof: The Cl\'ement-Li theorem requires the Lipschitz continuity of the function 
$$F(u) = -\frac{n-1}4 u^{\frac{n-5}{n-1}}R_{g_0} $$
as a map from a neighborhood of $u_0=1$ %in $(\cH^{s,\gamma}_p(\B),\cH^{s+2,\gamma+2}_p(\B))_{1-1/q,q}$ 
to $\cH^{s,\gamma}_p(\B)$. The scalar curvature satisfies 
$$xR_{g_0}= \frac1x (R_{h_0}-n(n-1) ) + \ \text{terms smooth in $x$ for $x\ge0$}.$$ 
In view of the assumption that $R_{h(0)} = n(n-1)$,  $R_{g_0}$ only blows up like $x^{-1}$ as $x\to0$. The  condition on $\gamma$, which is stronger than  in Theorem \ref{T5.1}, implies that $F$ maps into $\cH^{s,\gamma}_p(\B)$ and Lipschitz continuity follows from Lemmas \ref{si} and \ref{multiplier}.
\medskip

%%%%%%%%%%%%%%%%%%%%%%%%

\noindent{\bf The fractional porous medium equation.} Replacing $-\underline\Delta$ in Theorem \ref{T5.1} by its power $(-\underline\Delta)^\sigma$, $0<\sigma<1$, yields  the fractional porous medium equation 
\begin{eqnarray}
\label{FPME}
%\lefteqn{
\partial_t  u + (-\underline\Delta)^\sigma (u^m)=0,
\end{eqnarray}
studied by Roidos and Shao in \cite{RSh1} and \cite{RSh2}. 
Following Tanabe \cite[Section 2.3.2]{Tanabe79}  they define $(-\underline\Delta)_s^\sigma$ as an unbounded closed operator in $\cH^{s,\gamma}_p(\B)$ with the domain 
$$\scrD((-\underline\Delta)_s^\sigma) = \cH^{s+2\sigma, \gamma+2\sigma}_p\oplus \underline\scrE_0,$$
where  $\gamma$ is as in \eqref{eq.5.1} with $(n+1)/2-\gamma-2\sigma \not= q_j^\pm$ for all $j$.
This requires some care. In order to be able to apply the construction, they show that \eqref{sectorial} holds for
$-\underline\Delta$, also without the shift by a positive constant, since $\lambda=0$ is a simple pole of the resolvent. To this end they use a comparison with the Friedrichs extension $\underline\Delta_F$. 
\cite[Theorem 1.3]{RSh1} states:  

\begin{theorem}\label{TRS.1}
Fix $\gamma$ with $2\sigma>(n+1)/2-\gamma$, and  $1<p,q<\infty$ so large that 
$$\frac{n+1}p+\frac{2\sigma}q<2\sigma\quad \text{and} \quad \frac{n+1}2-\gamma-2\sigma +\frac{2\sigma}q<0.$$
Then there exists a $T>0$ such that the initial value problem for 
\eqref{FPME}, with  $u(0) = u_0\in  (\cH^{0, \gamma}_p(\B) ,\cH^{2\sigma, \gamma+2\sigma}_p(\B) \oplus\underline\scrE_0)_{1-1/q,q}$  strictly positive, has a unique solution 
\begin{eqnarray}
\label{eq.RS.2}
u\in W^1_q(0,T; \cH^{s,\gamma}_p(\B))\cap L^q(0,T; \cH^{s+2\sigma,\gamma+2\sigma}_p(\B) \oplus \underline\scrE_0) 
\end{eqnarray} 
for $s=0$. If $u_0\in \cup_{\varepsilon>0} \cH^{\nu+2+(n+1)/p+\varepsilon,\max\{\gamma+2,(n+3)/2\}+\varepsilon}_p(\B) \oplus \underline\scrE_0$, for some $\varepsilon>0$, $\nu\ge0$, then  \eqref{eq.RS.2} holds  for $s=\nu$. 
\!Moreover, $u\in \!C^\infty({]0,T[}, \cH^{2\sigma, \gamma+2\sigma}_p(\B))$.
\end{theorem} 

In \cite{RSh2} the authors omitted the positivity condition on $u$ and considered the problem
\begin{eqnarray}\label{5.15}
\partial_t u + (-\Delta)^\sigma (|u|^{m-1}u)=0, \quad u(0) = u_0,
\end{eqnarray}
with $m>0$ and $0<\sigma \le 1$. Their result then is, see \cite[Theorems 7.2, 7.3, 8.2]{RSh2}:

\begin{theorem}\label{TRS.2}
{\rm (a)} For $m\not=1$, $u_0\in L^\infty(\B)$ and any $T>0$,  \eqref{5.15}  has a unique strong solution. 
%in $C({[0,T[}; L^1(\B)$ for every $T>0$ that depends continuously on $u_0$, see \cite[Theorem 5.8]{RSh2}.

\noindent{\rm (b)} For $m=1$, $1<p<\infty$, $u_0\in L^p(\B)$ and any $T>0$, \eqref{5.15} has a unique solution 
$$u\in   C^1({[0,T[};L^p(\B))\cap C^0({[0,T[}; \scrD((-\underline\Delta_{F,p})^\sigma)).$$
When $u_0\in L^\infty(\B)$, this solution is strong and depends continuously on $u_0$ in the norm of  $C({[0,T[};L^1(\B))$.

\noindent{\rm (c)} Also, one has a   comparison principle, $L^p$ contractivity, and conservation of mass.  
\end{theorem} 

In Theorem \ref{TRS.2}(b), $\underline\Delta_{F,p}$ denotes the extension of $\Delta$ in $L^p(\B)$ induced by the Friedrichs extension $\underline \Delta_F$ in $L^2(\B)$. Theorem \ref{TRS.2}  requires different methods. The approach  is based on the Markovian properties of $\underline\Delta_F$.  

%%%%%%%%%%%%%%%%%%%%%%%%
\section{Example: The Cahn-Hilliard Equation} 
%%%%%%%%%%%%%%%%%%%%%%%%

The Cahn-Hilliard equation has been considered in \cite{RS4} and \cite{RS1}. 
Simplifying the arguments given there, we start by  rewriting the problem in the form 
\begin{eqnarray}\label{F} 
\partial_t u +\Delta^2u = F(u),  \quad F(u) =-\Delta(u-u^3).
\end{eqnarray}

\subsection{Short time solutions}
 
As  in Theorem \ref{T5.1} we fix $s\ge0$, $1<p<\infty$  and $\gamma\in \R$ such that 
\begin{eqnarray}\label{6.3}
\max\{-2,q_1^-\} <({n+1})/2 - \gamma-2<0
\end{eqnarray}
and $(n+1)/2-\gamma\not= q_j^+$ for all $j$. We consider the extension $\underline\Delta$  in $\cH^{s,\gamma}_p(\B)$ with  domain 
\begin{eqnarray}\label{6.4} 
\scrD(\underline\Delta) \!= \cH^{s+2,\gamma+2}_p(\B)\oplus \underline\scrE_0; \; \underline\scrE_0 \!=\!\{u\!: \!u(x,y) \!=\!\omega(x) e(y) ; e\!\in \!E_0\}.
\end{eqnarray}
%for some $1<p<\infty$ and, $\underline\scrE_0 =\{u: u(x,y) =\omega(x) e(y) ; e\in E_0\}$. 
By Theorem \ref{T4.8}, 
$c-\underline\Delta$ has a bounded $H_\infty$-calculus in $\cH^{s,\gamma}_p(\B)$ for sufficiently large $c>0$ and any positive angle. 
So it also has bounded imaginary powers of any positive angle. 
It follows from \cite[Lemma 3.6]{RS4} that also $\underline \Delta^2$ with its natural domain
\begin{eqnarray}\label{6.5} 
\scrD(\underline\Delta^2)=\{u\in \scrD(\underline \Delta): \Delta u \in \scrD(\underline\Delta)\}
\end{eqnarray}
has bounded imaginary powers of any positive angle after a shift by a positive constant. This is essentially a consequence of the fact that $(A^2)^{is} = A^{2is}$, $s\in \R$.
%, when $A$ has bounded imaginary powers.  
So  $\underline\Delta^2$ has maximal regularity. By \cite[Theorems 4.3 and 4.5 ]{RS4} and \cite[Theorem 5.9]{RS1} we obtain:

\begin{theorem}\label{T6.1} Let $s\ge0$, $p\ge n+1$, $q>2$,  $\gamma$ as in \eqref{6.3}, assuming additionally that $(n+1)/2-\gamma-4\not=q_j^-$ for all $j$. 
For the extension $\underline\Delta$ of the Laplacian in $\cH^{s,\gamma}_p(\B)$ with domain \eqref{6.4} let $\underline\Delta^2$ be the unbounded operator in $\cH^{s,\gamma}_p(\B)$ acting like $\Delta^2$ on the domain $\scrD(\underline\Delta^2)$ in \eqref{6.5}.
Given any  $u_0\in (\cH^{s,\gamma}_p(\B), \scrD(\underline\Delta^2))_{1-1/q,q}$, the initial value problem for the Cahn-Hilliard equation   
\begin{eqnarray*}
\partial_tu +\Delta^2u +\Delta(u-u^3) =0, \quad u(0) = u_0,
\end{eqnarray*}
has a unique solution 
\begin{eqnarray*}
u\in W^1_q(0,T;\cH^{s,\gamma}_p(\B))\cap L^p(0,T; \scrD(\underline\Delta^2)).
\end{eqnarray*}
\end{theorem} 

Idea of the {\em proof}. Let us first see  what $\scrD(\underline\Delta^2)$  is. 
Since $(n+1)/2-\gamma-4\not=q_j^-$ for all $j$, we know from Theorem \ref{T3.4} that  
$$\scrD(\Delta^2_{\min}) = \cH^{s+4,\gamma+4}_p(\B).$$

Next let us determine the maximal domain of $\Delta^2$ (not $\underline\Delta^2$), considered as an unbounded operator in $\cH^{s,\gamma}_p(\B)$.
According to \cite[(2.13)]{SS1} the principal Mellin symbol $\sigma_M(\Delta^2)$ of $\Delta^2$ satisfies 
\begin{eqnarray*}
\sigma_M(\Delta^2)(z) = \sigma_M(\Delta)(z+2)\sigma_M(\Delta)(z). 
\end{eqnarray*}
From \eqref{inverse} we obtain that
\begin{eqnarray*}
\sigma_M(\Delta^2)^{-1}(z)
% &=&   
%\sum_{j,k=0}^\infty\frac{1}{(z-q_j^+)(z-q_j^-)(z+2-q_k^+)(z+2-q_k^-)}\pi_j\pi_k\\
&=&\sum_{j=0}^\infty\frac{1}{(z-q_j^+)(z-q_j^-)(z+2-q_j^+)(z+2-q_j^-)}\pi_j,
\end{eqnarray*}
which allows us to read off the singularities including their multiplicities. Hence 
\begin{eqnarray*}
\scrD(\Delta^2_{\max}) = \cH^{s+4,\gamma+4}_p(\B) \oplus \bigoplus_{\rho \in J_\gamma} \scrE_\rho,
\end{eqnarray*}
where the direct sum is over all the asymptotics spaces associated with the singularities of $\sigma_M(\Delta^2)^{-1}$ in the interval
\begin{eqnarray*}
J_\gamma ={ \left]\frac{n+1}2-\gamma-4, \frac{n+1}2-\gamma\right[}.
\end{eqnarray*}
Since we defined the domain of $\underline\Delta^2$ by \eqref{6.5} we have to exclude the asymptotics spaces not contained in $\scrD(\underline\Delta)$ and conclude that 
\begin{eqnarray*}
\scrD(\underline\Delta^2) = \cH^{s+4,\gamma+4}_p(\B) \oplus \bigoplus_{J_\gamma'} \scrE_\rho,\oplus \,\underline\scrE_0,
\end{eqnarray*}
where $J_\gamma' = J_\gamma\setminus {[(n+1)/2-\gamma-2,(n+1)/2-\gamma[}$.

%Will, for the above choices of $p$ and $q$, an element 
%$u_0 \in (\cH^{0,\gamm}_p(\B),\scrD(\underline\Delta^2)_{1-1q,q}$ be a bounded function on $\B$, so that we can apply

Next let us have a look at the interpolation space 
$%\begin{eqnarray*}
(\cH^{s,\gamma}_p(\B), \scrD(\underline\Delta^2))_{1-1/q,q}.
$ %\end{eqnarray*}
Since $q>2$ by assumption, (an analog of) Lemma \ref{interpolation} shows that it embeds into
\begin{eqnarray*}
\lefteqn
{\bigcap_{\varepsilon >0}\cH^{s+4-4/q-\varepsilon,\gamma+4-4/q-\varepsilon }_p(\B) \oplus \bigoplus_\rho \scrE_\rho \oplus  \underline\scrE_0}\\
&&\hookrightarrow \cH^{s+2+\delta,\gamma+2+\delta}(\B) \oplus \bigoplus_\rho \scrE_\rho \oplus  \underline\scrE_0
\end{eqnarray*}
for  some $\delta>0$, where now the direct sum runs only over those $\scrE_\rho$ for which 
$$\rho \in \left[\frac{n+1}2-\gamma-4+\frac4q, \frac{n+1}2-\gamma -2\right[.$$
From Lemma \ref{2.8} we see that all these asymptotics spaces are contained in 
$\cH^{s+2,\gamma+2}_p(\B)$. As a result, the interpolation space embeds into $\cH^{s+2,\gamma+2}_p(\B)\oplus \underline \scrE_0$.  
We moreover observe that, by assumption $s+2>(n+1)/p$ and, by \eqref{6.3},  $\gamma+2>(n+1)/2 $.
Hence $\cH^{s+2,\gamma+2}_p(\B)\oplus \underline\scrE_0$ is a Banach algebra up to an equivalent norm by  Corollary \ref{algebra}.

Now for the existence of the solution.   The Cl\'ement-Li Theorem will imply the assertion of the theorem, provided we show that   the map $F$ defined in \eqref{F} is Lipschitz continuous  from a neighborhood $U$ of $u_0$ to $\cH^{s,\gamma}_p(\B)$. 

Since $\Delta:\cH^{s+2,\gamma+2}_p(\B)\oplus \underline\scrE_0 \to \cH^{s,\gamma}_p(\B)$ is bounded, we only have to show the Lipschitz continuity of 
\begin{eqnarray*}
u-u^3: \cH^{s+2,\gamma+2}(\B) \oplus\underline\scrE_0 \to \cH^{s+2,\gamma+2}_p(\B)\oplus \underline{\scrE}_0.
\end{eqnarray*}
So let $u_1,u_2\in \cH^{s+2,\gamma+2}(\B) \oplus  \underline\scrE_0$. Write 
\begin{eqnarray*}
(u_1-u_1^3) - (u_2-u_2^3) = (1-u_1^2-u_1u_2-u_2^2)(u_1-u_2). 
\end{eqnarray*}
Since $\cA := \cH^{s+2,\gamma+2}_p(\B)\oplus \underline\scrE_0 $ is a Banach algebra, we can estimate 
\begin{eqnarray*}
\lefteqn{\|(u_1-u_1^3) - (u_2-u_2^3) \|_{\cA}}\\
 &\le& C(1+\|u_1\|_{\cA}^2+\|u_1\|_\cA \|u_2\|_\cA + \|u_2\|^2_\cA ) \|u_1-u_2\|_\cA. 
\end{eqnarray*}
This concludes the argument. \hfill $\Box$

\subsection{Global Solutions and Attractors}
In \cite{LR1}  Lopes and Roidos showed:

\begin{theorem}\label{LoRo1}
Let $1<p<\infty$, $q\ge2$ for $p=2$ and $q>2$ for $p\not=2$, $s\ge0$, $s+2\ge(n+1)/p$. Choose $\scrD(\underline\Delta^2)$  as above with $\gamma$ as in \eqref{6.3}. Then the solution obtained in Theorem \ref{T6.1} belongs to $ C^\infty({]0,T[},\scrD(\underline{\Delta}_\sigma^2),$ for all $\sigma$, where $\scrD(\underline{\Delta}_\sigma^2)$ is the corresponding domain for $\underline{\Delta}^2$ in $\cH^{\sigma,\gamma}_p(\B)$.  

If $n=1$ and $\gamma<-1/2$ or $n=2$ and $\gamma<-1/4$, then the solution exists for all times. 
\end{theorem} 

That the solutions become instantaneously smooth in time and space follows essentially from the fact that one can take any $t_1\in {]0,T[}$ and find the solution with $u(t_1)$ as initial value. On the one hand, it coincides with the previously found solution, on the other it has better regularity. Then one  applies an iteration argument, see \cite[Theorem 3.1]{RS3}. 
 
 The long time existence is a consequence of two interesting facts. The first is \cite[Corollary 4.2]{LR1}. It shows that, if one has a global solution for one choice of $s,p,q$ as above, then also for any other. The second is an energy estimate, see \cite[Proposition 4.4]{LR1}: For some $C$, depending only on $\|\nabla u_0\|_{L^2(\B, d\mu_g)} $ and $\|u^2_0-1\|_{L^2(\B,d\mu_g)}$,
 %One has $u\in C([0,T]; \cH_2^{1,1+\gamma}(\B)$ and 
\begin{eqnarray*}
\|u\|_{ C([0,T]; \cH_2^{1,1+\gamma}(\B))} \le C,
\end{eqnarray*}
independent of $T$. This estimate is derived by studying the energy functional 
$$\Phi(u) = \frac12\|\nabla u\|^2_{L^2(\B,d\mu_g)} + \frac14\|u^2-1\|^2_{L^2(\B,d\mu_g)} .$$

As shown in \cite[Theorem 1]{LR2}, the global solutions in Theorem \ref{LoRo1} have a striking property. Let $s\ge0$, $\gamma$ as in \eqref{6.3}, and let  $X^s_{1,0}$ be the set of all real-valued $u$ in $\cH^{s+2,\gamma+2}_2(\B)\oplus \underline\scrE_0$ with $\int_\B u\,d\mu_g=0$. 

\begin{theorem}{\rm(a)} The semiflow $\T: {[0,\infty[} \times X^s_{1,0}\to X^s_{1,0}$ defined by $\T(s,u_0)=: \T(s)(u_0) = u(s)$  
has an $s$-independent global attractor $\mathbb A\subset \cap_{\sigma>0} \scrD(\underline{\Delta}^2_\sigma)$. If $B\subset \scrD(\underline{\Delta}^2_s)$ is bounded in $X^s_{1,0}$, then, for any $\sigma>0$,   $\T(t)(B)$ is bounded in $\scrD(\underline{\Delta}^2_\sigma)$ for large $t$ and 
\begin{eqnarray*}
\lim_{t\to \infty}(\sup_{b\in B}\inf_{a\in \mathbb A} \|\T(t)b-a\|_{\scrD(\underline\Delta^2_\sigma)} )=0.
\end{eqnarray*}
{\rm(b)} For $u_0\in X_{1,0}^s$ there exists a $u_\infty\in \cap_\sigma\scrD(\underline\Delta^2_\sigma)$ with 
$\T(t)u_0\to u_\infty $ in $\scrD(\underline\Delta^2_\sigma)$ for all $\sigma$. 
\end{theorem}

%%%%%%%%%%%%%%%%%%%%%%%%

\end{document}